\documentclass[a4paper,twoside,reqno]{amsart}
\usepackage[utf8]{inputenc}

\usepackage{amsmath,amssymb,amsfonts,amsthm,enumerate,mathtools,tikz}

\usepackage{fullpage}
\usepackage{hyperref}
\renewcommand{\MR}[1]{\href{http://www.ams.org/mathscinet-getitem?mr=#1}{MR#1}}

\numberwithin{equation}{section}

\newcommand{\Span}{\mathrm{Span}}
\newcommand{\Nor}{\mathrm{Nor}}

\newcommand{\C}{\mathbb{C}}
\newcommand{\cA}{\mathcal{A}}

\newcommand{\cT}{\mathcal{T}}
\newcommand{\cO}{\mathcal{O}}

\newcommand{\rV}{\mathrm{v}}

\newcommand{\cR}{\mathcal{R}}
\newcommand{\rS}{\mathrm{s}}

\newcommand{\Ms}{\mathrm{MS}}
\newcommand{\Mv}{\mathrm{MV}}
\newcommand{\asX}{\langle X_n \rangle}
\newcommand{\asoX}{\langle X \rangle}
\newcommand{\asY}{\langle Y_N \rangle}
\newcommand{\asZ}{\langle Z \rangle}
\newcommand{\asXY}{\langle X;Y \rangle}

\newcommand{\p}{\mathbb{P}}

\newcommand{\N}{\mathbb{N}}

\def\gkdim{\operatorname {GK dim}}
\def\gldim{\operatorname {gl\,dim}}
\def\deg{\operatorname {deg}}

\theoremstyle{plain}

\newtheorem{thm}{Theorem}[section]
\newtheorem{pro}[thm]{Proposition}
\newtheorem{lem}[thm]{Lemma}
\newtheorem{cor}[thm]{Corollary}

\theoremstyle{definition}

\newtheorem{notaz}[thm]{Notation}

\newtheorem{ex}[thm]{Example}

\newtheorem{facts}[thm]{Facts}

\newtheorem{convention}[thm]{Convention}
\newtheorem{connotation}[thm]{Notation}
\newtheorem{dfn}[thm]{Definition}

\newtheorem{rmk}[thm]{Remark}

\newcommand{\im}[0]{\operatorname{im}}

\newcommand{\q}{\mathbf{q}}
\newcommand{\g}{\mathbf{g}}

\newcommand{\LM}{\mathbf{LM}}

\newcommand{\mas}{multiplicatively anti-symmetric\ }

\title[Veronese and Segre morphisms]{Veronese and Segre morphisms between non-commutative projective spaces}

\author{Francesca Arici}

\address{Mathematical Institute, Leiden University, P.O. Box 9512, 2300 RA
Leiden, the Netherlands}
\email{f.arici@math.leidenuniv.nl}

\author{Francesco Galuppi}
\address{Max Planck Institute for Mathematics in the Sciences, Inselstraße 22, 04013 Leipzig, Germany}
\email{galuppi@mis.mpg.de}

\author{Tatiana Gateva-Ivanova}
\address{American University in
Bulgaria, 2700 Blagoevgrad, Bulgaria} \email{tatyana@aubg.edu}

\thanks{FA was partially funded by the Max Planck Institute for Mathematics in the Sciences, (MPI MiS), Leipzig, and the Netherlands Organisation of Scientific Research (NWO) under the VENI grant 016.192.237.
TGI was partially supported by the Max Planck Institute for Mathematics, (MPIM), Bonn, by MPI MiS, Leipzig,
and by Grant KP-06 N 32/1 of 07.12.2019 of the Bulgarian National Science Fund.
}
\begin{document}
\date{\today}
\begin{abstract}
We study Veronese and Segre morphisms between non-commutative projective spaces. We compute finite reduced Gr\"obner bases for their kernels, and we compare them with their analogues in the commutative case.\end{abstract}
\maketitle

\section{Introduction}\label{sec:intro}
In this work, we describe Veronese and Segre morphisms for a class of non-commutative quadratic algebras which have permeated the literature under different names. They appeared as \emph{quantum affine spaces} in \cite[Section 1 and
Section 4]{Ma87}, and more recently as \emph{non-commutative projective spaces} in the work \cite{AuKaOr08} on mirror symmetry, as well as in the study of deformations of toric
varieties \cite{CiLaSz11,CiLaSz13}.

Motivated by the interpretation of morphisms between non-commutative algebras as "maps between non-commutative spaces", we consider here non-commutative analogues of the Veronese and Segre embeddings, two fundamental maps that play pivotal roles not only in
classical algebraic geometry, but also in applications to other fields of mathematics.

The $d$-Veronese map is the non-degenerate embedding of the projective space $\p^n$ via the very ample line bundle $\cO(d)$. Its image, called the Veronese variety, has a capital
importance in algebraic geometry. Just to mention an example, every projective variety is isomorphic to the intersection of a Veronese variety and a linear space (see
\cite[Exercise 2.9]{Harris}). The Segre map is the embedding of $\p^m\times\p^n$ via the very ample line bundle $\cO(1,1)$. It is used in projective geometry to endow the Cartesian product of two projective spaces with the structure of a projective variety. In quantum mechanics and quantum information theory, it is a natural mapping for describing non-entangled states (see \cite[Section 4.3]{BeZy}).  Both are studied for the theory of tensor decomposition \cite[Section 4.3]{Lan}, as the image of the Segre morphism is the locus of rank 1 tensors,
 while the image of the Veronese morphism plays
a similar role for symmetric tensors. Moreover, these constructions are central
in the field of algebraic statistics: the variety of moments of a Gaussian random variable is a Veronese variety
(see \cite[Section 6]{AFS}), while independence models are encoded by Segre varieties (see \cite{GaStiSturm05}).

 The natural problem of finding
non-commutative counterparts of those fundamental constructions has been addressed from different perspectives, for instance in \cite{V92} and
\cite{Sm03}.

In this work, we study the properties of these maps and of the corresponding algebras from the point of view of the theory of  Gr\"obner bases. In classical
algebraic geometry, a variety $V$ is completely determined by its defining ideal. When $V$ is the image of a variety morphism $f$,
the ideal of $V$ is the kernel of the algebra morphism corresponding to $f$. Computing a Gr\"obner basis for the defining ideal can provide
valuable information about the
properties of $V$. With this motivation in mind, we are interested in computing Gr\"obner bases for the kernels of the non-commutative Veronese and Segre morphisms.

The paper is structured as follows. In Section~\ref{sec:pre}  we recall some basics of
the theory of Gr\"obner bases for ideals in the free associative algebra.
Our Lemma \ref{lem:diamondquadratic} gives a criterion for quadratic Gr\"obner bases, which is crucial for the proof of our main results, Theorems \ref{thm:Veronese_ker} and
\ref{thm:Segre_ker}.
In Section~\ref{sec:nc_projective_spaces}
we present the quadratic algebras $\cA=\cA^n_\q$,
called \emph{quantum spaces}, or
\emph{non-commutative projective spaces},
and we recall some of their basic properties.
 In Section
\ref{sec:dVeronese} we analyse their $d$-Veronese subalgebras. The main result of the section is
 Theorem \ref{thm:d-Veronese_relations}, which gives a presentation of the $d$-Veronese subalgebra
 in terms of generators and quadratic relations.
 In Section~\ref{sec:Veronese}
 we introduce and study non-commutative analogues of the Veronese maps for non-commutative projective spaces.
 We present a modification of
the theory of Gr\"obner
bases for ideals in a quantum space 
and find explicitly
a Gr\"obner basis for the kernel of the Veronese map in Theorem  \ref{thm:Veronese_ker}.
Using a similar approach and methods, in Section~\ref{Sec:Segremap} we introduce and study non-commutative analogues of Segre maps and Segre products.
Theorem \ref{thm:Segre_ker} describes the reduced Gr\"{o}bner basis for the kernel of the Segre map.
Finally, in Section \ref{sec:examples} we present various examples that illustrate our results.

\subsection*{Acknowledgements} FA and FG would like to thank the Max Planck Institute for Mathematics in Bonn for their hospitality during a visit to TGI in May 2019. We would also like to thank Bernd Sturmfels for encouraging this line of research and Giovanni Landi for discussions on toric non-commutative manifolds.


\section{Preliminaries}
\label{sec:pre}
 We start with notation, conventions, and facts which will be used throughout the paper,
and recall some basics on Gr\"obner bases for ideals in the free associative algebra. Lemma \ref{lem:diamondquadratic} gives a criterion
for quadratic Gr\"obner bases which is particularly useful in our settings.

\subsection{Basic notations and conventions}

Throughout the paper $X_n= \{x_0, \dots, x_n\}$ denotes a non-empty set of indeterminates. To simplify notation, we shall often write $X$ instead of $X_n$. We denote by $\C\langle x_0, \dots, x_n\rangle$ the complex free associative algebra with unit generated by $X_n$, while $\C [X_n]$
denotes the commutative polynomial ring in the variables $x_0,\dots,x_n$.
 $\asX$ is the free monoid generated by $X_n$, where the unit is the empty word, denoted by $1$.

We fix the degree-lexicographic order $<$ on $\asX$, where we set $x_0 < x_1< \dots <x_n$. As usual, $\N$ denotes the set of all positive integers,
and $\N_0$ is the set of all non-negative integers.
Given a non-empty set $F\subset \C \asX$, we write
    $(F)$ for the two-sided ideal of $\C\asX$ generated by $F$.

In more general settings, we shall also consider associative algebras over a field $\textbf{k}$.
Suppose $A= \bigoplus_{m\in\N_0}  A_m$ is a graded $\textbf{k}$-algebra
such that $A_0 =\textbf{k}$, and such that $A$ is finitely generated by elements of positive degree. Recall that its Hilbert function is $h_A(m)=\dim A_m$
and its Hilbert series is the formal series $H_A(t) =\sum_{m\in\N_0}h_{A}(m)t^m$.
In particular, the algebra $\C [X_n]$ of commutative polynomials satisfies
\begin{equation}
\label{eq:hilbert}
h_{\C [X_n]}(d)= \binom{n+d}{n} \quad\mbox{and}\quad  H_{\C [X_n]}= \frac{1}{(1 -t)^{n+1}}.
\end{equation}

We shall use two well-known gradings on the free associative algebra $\C\asX$: the \emph{natural grading by length} and the $\N_0^{n+1}$\emph{-grading}.

Let $X^m$ be the set of all words of length $m$ in $\asoX$. Then
\[\asoX = \bigsqcup_{m\in\N_0}  X^{m},\
X^0 = \{1\},\ \mbox{and} \ \   X^{k}X^{m} \subseteq X^{k+m},\]
so the free monoid $\asoX$ is naturally \emph{graded by length}.

Similarly, the free associative algebra $\C\asoX$ is also graded by length:
\[\C\asoX
 = \bigoplus_{m\in\N_0} \C\asoX_m,\quad \mbox{ where}\ \  \C\asoX_m=\C X^{m}. \]

A polynomial $f\in  \C\asoX$ is \emph{homogeneous of degree $m$} if $f \in \C X^{m}$.
We denote by
\[\cT^n =\cT(X_n) :=\left\lbrace x_0^{\alpha_0}\cdots x_n^{\alpha_n}\in\asX \ \vert \ \alpha_i\in\N_0, i\in\{0,\dots,n\}\right\rbrace\]
the set of ordered monomials (terms) in $\asX$   and by
\[
\cT^n_d = \cT_d(X_n):=\left\lbrace x_0^{\alpha_0}\cdots x_n^{\alpha_n}\in\cT^n \mid\ \sum_{i=0}^n \alpha_i = d \right\rbrace
\]
the set of ordered monomials of length $d$.
It is well known that the cardinality
$|\cT_d(X_n)|$ is given by the Hilbert function (Hilbert polynomial) $h_{\C [X_n]}(d)$ of the polynomial ring in the variables $X_n$:
\begin{equation}
        \label{eq:order-deg-d-monomials}
|\cT_d(X_n)|= \binom{n+d}{n}=h_{\C [X_n]}(d).
\end{equation}

\begin{dfn}
\label{def:multidegreedef}
 A monomial $w \in \asoX $ has
  \emph{multi-degree} $\alpha =(\alpha_0, \dotsc, \alpha_n) \in \N_0^{n+1}$, if $w$,
considered as a commutative term,
can be written as $w = x_0^{\alpha _0}x_1^{\alpha _1}\dotsm x_n^{\alpha _n}$. In this case we write $\deg (w) =\alpha$.
Clearly, $w$ has length $|w|= \alpha _0  +\dotsb + \alpha _n.$
In particular, the unit $1\in \asoX$ has multi-degree $\textbf{0} =(0, \dotsc, 0)$, and $\deg(x_0) = (1,0,\dotsc, 0),   
\dotsc, \deg(x_n)= (0,0,\dotsc, 1)$. For each $\alpha= (\alpha _0, \alpha _1,\dotsc, \alpha _n) \in \N_0^{n+1}$
we define 
\begin{equation}
        \label{eq:multigrading2}
\begin{array}{c}
T_{\alpha}:= x_0^{\alpha _0}x_1^{\alpha _1}\dotsm x_n^{\alpha _n} \in \cT(X_n) \quad \mbox{and}\quad
X_{\alpha}: = \{w \in  \asoX \mid \deg(w)= \alpha \}.
\end{array}
\end{equation}
The free monoid $\asX$ is naturally $\N_0^{n+1}$-graded:
\[\asX = \bigsqcup_{\alpha\in \N_0^{n+1}}  X_{\alpha}, \; \mbox{where}\;
X_{\textbf{0}} = \{1\}, \;\mbox{and }\;  X_{\alpha}X_{\beta}\subseteq X_{\alpha+\beta}.\]

In a similar way, the free associative algebra $\C\langle X_n\rangle$ is also canonically $\N_0^{n+1}$ -graded:
\[\C \asX = \bigoplus_{\alpha \in \N_0^{n+1}} \C\langle X_n\rangle_{\alpha},  \text{ where }  \C\langle X_n\rangle_{\alpha}= \C  X_{\alpha}.\]
\end{dfn}

It follows straightforwardly from (\ref{eq:multigrading2}) that
$X_{\alpha}\cap \cT(X_n)  = \lbrace T_{\alpha} \rbrace$,
for every $\alpha \in \N_0^{n+1}$.
Moreover, every $u \in X_{\alpha} \setminus \{T_{\alpha}\}$ satisfies $u > T_{\alpha}$, i.e., $T_{\alpha}$ is the \emph{minimal element} of $X_{\alpha}$
with respect to the ordering $<$.

\subsection{Gr\"obner bases for ideals in the free associative algebra}
\label{sec:grobner}
In this subsection $\textbf{k}$ is an arbitrary field and $X= X_n=\{x_0,\dotsc,x_n\}$.
Suppose $f \in \textbf{ k}\asoX$ is a nonzero polynomial. Its leading
monomial with respect to $<$ will be denoted by $\LM(f)$.
One has $\LM(f)= u$ if
$f = cu + \sum_{1 \leq i\leq m} c_i u_i$, where
$ c,c_i \in \textbf{k}$, $c \neq 0 $ and $u > u_i\in \asoX$, for every $i\in\{1,\dots,m\}$.

Given a set $F \subseteq \textbf{k} \asoX$ of
non-commutative polynomials, $\LM(F)$ denotes the set
 \[\LM(F) = \{\LM(f) \mid f \in F\}.\]
A monomial $u\in \asoX$ is \emph{normal modulo $F$} if it does not contain any of the monomials $\LM(f)$ as a subword.
 The set of all normal monomials modulo $F$ is denoted by $N(F)$.

Let  $I$ be a two sided graded ideal in $K \asoX$ and let $I_m = I\cap \textbf{k}X^m$.  We shall consider
graded algebras with a minimal presentation. Without loss of generality, we may assume that
$I$ \emph{is generated by homogeneous polynomials of degree $\geq 2$}
and $I = \bigoplus_{m\ge 2}I_m$. Then the quotient
algebra $A = \textbf{k} \asoX/ I$ is finitely generated and inherits its
grading $A=\bigoplus_{m\in\N_0}A_m$ from $ \textbf{k} \asX$. We shall work with
the so-called \emph{normal} $\textbf{k}$-\emph{basis of} $A$.

We say that a monomial $u \in \asX$ is  \emph{normal modulo $I$} if it is normal modulo $\LM(I)$. We set
$N(I):=N(\LM(I))$.
In particular, the free
monoid $\asoX$ splits as a disjoint union
\begin{equation}
\label{eq:X1eq2a}
\asoX=  N(I)\sqcup \LM(I).
\end{equation}
The free associative algebra $\textbf{k} \asoX$ splits as a direct sum of $\textbf{k}$-vector
  subspaces
  $\textbf{k} \asoX \simeq  \Span_{\textbf{k}} N(I)\oplus I$,
and there is an isomorphism of vector spaces
\begin{equation}
A \simeq \Span_{\textbf{k}} N(I).
\end{equation}

We define
\[N(I)_{m}=\{u\in N(I)\mid u\mbox{ has length } m\}.\]
 Then
$A_m \simeq \Span_{\textbf{k}} N(I)_{m}$ for every $m\in\N_0$.

\begin{dfn}
\label{eq: s.f.p.def} Let $I\subset\textbf{k}\asX$ be a two-sided ideal.
\begin{enumerate}
\item A subset
$G \subseteq I$
of monic polynomials is a \emph{Gr\"{o}bner
basis} of $I$ (with respect to the ordering $<$) if
\begin{enumerate}
\item $G$ generates $I$ as a
two-sided ideal, and
\item for every $f \in I$ there exists $g \in G$ such that $\LM(g)$ is a subword of $\LM(f)$, that is
$\LM(f) = a\LM(g)b$,  for some $a, b \in \asoX$.
\end{enumerate}
\item
A Gr\"{o}bner basis $G$ is \emph{minimal} if  the set $G\setminus\{f\}$ is not a Gr\"{o}bner basis of $I$, whenever $f \in G$.

\item
A minimal Gr\"{o}bner basis $G$ of
$I$ is \emph{reduced} if each  $f \in G$  is a linear combination of normal monomials modulo $G\setminus\{f\}$. In this case we say that $f$ is \emph{reduced modulo}
$G\setminus\{f\}$.

\item
If $I$ has a finite Gr\"{o}bner basis $G$, then the algebra $A = \textbf{k}\asoX / (G)$ is
called a \emph{standard finitely presented algebra}, or shortly an \emph{s.f.p. algebra}.
\end{enumerate}
\end{dfn}
It is well-known that every ideal $I$ of $\textbf{k} \asoX$ has a unique reduced
Gr\"{o}bner basis $G_0= G_0(I)$ with respect to $<$. However, $G_0$ may be infinite.
 For more details, we refer the reader to \cite{Latyshev,Mo94, EPS}.

\begin{dfn}\label{def: reductions}
    Let $h_1,\dots,h_s\in \textbf{k} \asoX$  ($h_i = 0$ is also possible).
 For every $i\in\{1,\dots, s\}$, let $w_i \in \asoX$ be a monomial of degree at least 2, such that $w_i>\LM(h_i)$, whenever $h_i \neq 0$, and let
$g_i=w_i-h_i$.
Each $g_i$ is a monic polynomial with $\LM(g_i) = w_i$.
    Let $G = \{g_1\dots, g_s\}\subset \textbf{k}\asoX$ and let $I = (G)$ be the two-sided ideal of $\textbf{k}\asoX$ generated by $G$.
For $u,v\in\asoX$ and for $i\in\{1,\dots,s\}$, we consider the $\textbf{k}$-linear operators $r_{uiv}:\textbf{k}\asX\to \textbf{k}\asX$ called \textit{reductions}, defined on the basis elements $c\in \asX$ by
\[
r_{uiv}(c)=   \begin{cases}
           uh_iv                                 & \text{if } c = uw_iv\\
           c                                     & \text{otherwise}.
                                 \end{cases}
\]
Then the following conditions hold:
\begin{enumerate}
    \item $c - r_{uiv}(\omega) \in I$.
    \item $\LM(r_{uiv}(c)) \leq c$.
    \item More precisely, $\LM(r_{uiv}(c)) < c$ if and only if $c = uw_iv$.
\end{enumerate}
More generally, for $f \in \textbf{k}\asoX$ and for any finite sequence of reductions $r= r_{u_1i_1v_1}\circ \dots \circ r_{u_ti_tv_t}$ one has
\[f\equiv r(f)  (\text{mod}   I)   \text{ and }   \LM(f) \geq \LM(r(f)).\]
A polynomial $f \in \textbf{k}\asX$ is in \emph{normal form (mod $G$)} if none of its monomials contains as a subword any of the $w_i$'s. In particular, the $0$ element is in normal form.
\end{dfn}
The degree-lexicographic ordering $<$ on $\asX$ satisfies the decreasing chain condition, and therefore for every $f \in \textbf{k}\asoX$  one can find a normal form of $f$
by means of a finite sequence of reductions defined via $G$. In general, $f$ may have more than one normal forms (mod $G$).
 It follows from Bergman's Diamond Lemma (see \cite[Theorem 1.2]{Bergman})  that $G$ is a Gr\"{o}bner basis of $I$  if and only if  every $f \in \textbf{k}\asoX$
has a unique normal form (mod $G$), which will be denoted by $ \Nor(f)$. In this case $f \in I$ if and only if $f$ can be reduced to $0$ via a finite sequence of reductions.

\begin{dfn}
\label{dfn:compositions_def} Let $G = \{g_i= w_i -h_i\mid i\in\{1,\dots, s\}\}\subset \textbf{k}\asX$ as in Definition \ref{def: reductions} and let $I=(G)$. Let $u=w_i$ and $v=w_j$ for some $i,j\in\{1,\dots, s\}$ and let $a,b,t \in \asoX \setminus \{1\}$.
\begin{enumerate}
\item
Suppose that $u=ab$, $v= bt$ and let
 $\omega = abt=ut=av$.
 The difference
\[(u,v)_{\omega}= g_it-ag_j= ah_j-h_it\]
 is called \emph{a composition of overlap}.
 Note that $(u,v)_{\omega}\in I$ and $\LM(g_it) =\omega= \LM(ag_j)$,
 so
\[
\LM((u,v)_{\omega}) = \LM(sh_j-h_it)< \omega.\]
The composition of overlap $(u,v)_{\omega}$ is \emph{solvable} if it can be reduced to $0$ by means of a finite sequence of reductions defined via $G$.

\item Suppose that $\omega= w_j =aw_ib$. The \emph{composition of inclusion} corresponding to the pair $(u,\omega)$ is
\[(u,\omega)_{\omega}:= (ag_ib)- g_j =h_j - ah_ib. \]
One has $(u, \omega)_{\omega} \in I$ and $\LM(u, \omega)_{\omega} = \LM(h_j - ah_ib)< \omega$.
The composition of inclusion $(u, \omega)_{\omega}$ is \emph{solvable} if it can be reduced to $0$ by means of a finite sequence of reductions defined via $G$.
\end{enumerate}
\end{dfn}

The lemma below is a modification of the Diamond Lemma and follows easily from Bergman's result \cite[Theorem 1.2]{Bergman}.

\begin{lem}
    \label{lem:diamonlemma}
Let $G = \{w_i -h_i\mid i\in\{1,\dots,s\}\}\subset \textbf{k}\asX$ as in Definition \ref{def: reductions}. Let $I = (G)$ and let $A = \textbf{k}\asX/I.$ Then the following
conditions are equivalent.
\begin{enumerate}
\item
The set
$G$  is a Gr\"{o}bner basis of $I$.
\item All compositions of overlap  and all compositions of inclusion are solvable.
\item Every element $f\in \textbf{k}\asX$ has a unique normal form modulo $G$, denoted by $\Nor_G (f)$.
\item
There is an equality $N(G) = N(I)$, so there is an isomorphism of vector spaces
\[\textbf{k}\asX \simeq I \oplus \textbf{k}N(G).\]
\item The image of $N(G)$ in $A$ is a $\textbf{k}$-basis of $A$.
In this case
$A$ can be identified with the $\textbf{k}$-vector space $\textbf{k}N(G)$, made a $\textbf{k}$-algebra by the multiplication
$a\bullet b: = \Nor(ab).$\end{enumerate}
Suppose furthermore that $G$ consists of homogeneous polynomials. Then $A$ is graded by length and each of the above conditions is equivalent to
\begin{enumerate}
\item[(6)] $\dim A_m = \dim (\textbf{k} N(G)_m) = |N(G)_m|$ for every $m \in\N_0$.
\label{lem:diamonlemma6}
\end{enumerate}
\end{lem}

\begin{cor}
\label{cor:diamonlemma}
Let $G = \{w_i -h_i\mid i\in\{1,\dots,s\}\}\subset \textbf{k}\asX$ as above and let $I=(G)$.
 Let $N(G)$ and $N(I)$
be the corresponding sets of normal monomials in $\textbf{k}\asX$. Then
$N(G) \supseteq   N(I)$, where an equality holds if and only if $G$ is a Gr\"{o}bner basis of $I$.
\end{cor}

 It is shown in \cite[Corollary 6.3]{KRW90} that there exist ideals in the free associative algebra $\textbf{k}\langle x_0, \dots, x_n\rangle$
 for which the existence of a finite Gr\"{o}bner basis is an undecidable problem.

 In this paper, we focus on a class of quadratic standard finitely presented algebras $\cA$ known as \emph{non-commutative projective spaces} or \emph{quantum spaces}.
Each such algebra $\cA$ is \emph{strictly ordered} in the sense of \cite[Definition 1.9]{GI91},
so there is a well-defined notion of Gr\"{o}bner basis of a two-sided
 ideal in $\cA$ (cf. \cite[Definition 1.2]{GI91}).
Moreover, every two-sided ideal in $\cA$ has a finite reduced Gr\"{o}bner basis.

\subsection{Quadratic algebras and quadratic Gr\"{o}bner  bases}
\label{subs:quadratic}
As usual, let $X= X_n=\{x_0,\dots,x_n\}$. Let $M$ be a non-empty proper subset of $\{0,\dots,n\}^2$. For every $(j,i)\in M$, let $h_{ji}\in\textbf{k}\asoX$ be either 0 or a homogeneous
 polynomial of degree 2 with
 $\LM (h_{ji})< x_jx_i$. Let
\begin{equation}
\label{defeq:R}
\cR = \{ f_{ji}= x_jx_i - h_{ji}\mid  (j,i) \in M\}  \subset \textbf{k}\asoX.
\end{equation}
Define $I=(\cR)$ and consider the quadratic algebra $A = \textbf{k}\asX/ I$.
As in Subsection \ref{sec:grobner}, let
$N(I)_m = N(I)\cap (X_n)^{m}$ and $N(\cR)_m = N(\cR)\cap (X_n)^{m}$ be the corresponding subsets of normal words of length $m$. By construction, $\cR$ is a $\textbf{k}$-basis for $I_2$, so
\[\dim I_2 = |\cR|= |M|\quad \mbox{and}\quad  N(I)_2 = N(\cR)_2 = X_n^2\setminus \LM(\cR).\]
As vector spaces,
\[
\textbf{k}\asoX = I \oplus \textbf{k}N(I)\ \mbox{and}\ A \cong \textbf{k} N(I).
\]
Moreover, for the canonical grading by length one has
 \[
 (\textbf{k}\asoX)_m = (I)_m \oplus \textbf{k}N(I)_m\ \mbox{and}\ A_m \cong \textbf{k} N(I)_m,
 \]
 for every $m\in\N$.

The following Lemma is crucial for the proofs of several results in the paper.

\begin{lem}
\label{lem:diamondquadratic}
 Let $\cR$ be defined as in \eqref{defeq:R}. The following conditions are equivalent.
\begin{enumerate}
\item
The set
$\cR$  is a (quadratic) Gr\"{o}bner basis of the ideal $I=(\cR)$.
\item
$\dim A_3 = |N(\cR)_3|$.
\item
All ambiguities of overlap determined by $\LM(\cR) = \{x_jx_i\mid (j,i) \in M \}$ are $\cR$-solvable.
\end{enumerate}
In this case $A$ is a \emph{PBW} algebra in the sense of
\cite[Section 5]{priddy}.
\end{lem}

\begin{proof}First note that there are no compositions of inclusions.
By Corollary \ref{cor:diamonlemma},
\[
N(I)_m \subseteq N(\cR)_m\mbox{ and }\dim A_m = |N(I)_m| \leq |N(\cR)_m |\]
for every $ m \geq 2$.  The implications  $(1)  \Longleftrightarrow  (3)$ and $(1)  \Rightarrow  (2)$
  follow from Lemma \ref{lem:diamonlemma}.

 $(2) \Rightarrow  (3)$  A composition of overlap is either $0$, or  it produces only homogeneous polynomials of degree three.
Suppose $\omega = x_kx_jx_i,$ where $(k, j), (j,i)\in M$, so $f_{kj}= x_kx_j - h_{kj} \in \cR$ and  $f_{ji}= x_jx_i - h_{ji} \in \cR.$
Then the corresponding composition of overlap is
\[
(x_kx_j, x_jx_i)_{\omega}= (f_{kj})x_i - x_k(f_{ji} )= -h_{kj}x_i+ x_k h_{ji} \in I.
\]
By Definition \ref{dfn:compositions_def}, a composition is solvable if and only if it can be reduced to $0$. Assume by contradiction that the composition $(x_kx_j, x_jx_i)_{\omega}$ is not solvable.
Then $(x_kx_j, x_jx_i)_{\omega}\neq 0$ and we can reduce it by means of a finite sequence of reductions
to a (not necessarily unique) normal form
$$F:= \Nor((x_kx_j, x_jx_i)_{\omega}) = cu + \sum_{s=1}^{t} c_su_s \in
\textbf{k}N(\cR),$$
 where $u > u_s$ and  $c \neq 0.$ In particular, $x_kx_jx_i >\LM(F) = u \in N(\cR)$.
  However, the polynomial $F$ is in the ideal $I$, hence $\LM(F) \in \LM(I_3)$ and $\LM (F)$   is not in $N(I)_3$.
  Therefore
   \[N(I)_3 \subsetneqq N(\cR)_3.\] Note that we have an isomorphism of vector spaces \[A_3 \cong\textbf{k} N(I)_3, \]
  hence $\dim A_3 = |N(I)_3 | < |N(\cR)_3|$, a contradiction.
\end{proof}

\begin{rmk}
\label{rmk:important}
 Lemma \ref{lem:diamondquadratic} is very useful for the case when we want to show that an algebra $A$ with explicitly given quadratic defining relations
$\cR \subset \textbf{k}\asX$ is \emph{PBW} (that is $\cR$ is a
Gr\"{o}bner basis of the ideal $I=(\cR)$) and we have precise
information about the dimension $\dim A_3 = d_3$. In this case,
instead of following the standard procedure (algorithm) of
checking whether all compositions are solvable,  we suggest a new
simpler procedure:
\begin{enumerate}
    \item find the set $N(\cR)_3$ and its order $|N(\cR)_3|$, and
    \item compare the order $|N(\cR)_3|$
    with $\dim A_3$.
\end{enumerate}
One has $|N(\cR)_3| \geq \dim A_3$ and an equality holds if and only if  $\cR$ is a Gr\"{o}bner basis of the ideal $I=(\cR)$.
 This method is particularly useful when we work in general settings--general $n$ and general quadratic relations
$\cR$.
It implies a similar procedure for ideals in the quantum space $\cA^N_\g$.

We use this result in Section \ref{sec:Veronese}, see the proof of Theorem \ref{thm:preVeronese_ker}.
In Subsection \ref{subsec:grobner_Quantspace} we give some basics on Gr\"{o}bner bases for ideals in a quantum space $\cA^N_\g$.
Lemma \ref{lem:quant_diamondquadratic}
is an important analogue of Lemma \ref{lem:diamondquadratic}  designed for quadratic Gr\"{o}bner bases of ideals in a quantum space.
\end{rmk}

\section{Quantum spaces}
\label{sec:nc_projective_spaces}

In this section we introduce a class of quadratic algebras which
are central for the paper. We shall refer to them as \emph{quantum
spaces}. They form a special case of the non-commutative
deformation of projective spaces defined by Auroux, Katzarkov, and
Orlov in the context of mirror symmetry \cite{AuKaOr08}. These
algebras are a particular case of the skew-polynomial rings with
binomial relations studied in \cite{Ga94,Ga96}.  We
point out that these objects appear with different
names in the literature: they are sometimes referred to as
\emph{non-commutative projective spaces} and \emph{quantum affine
spaces}.  We shall now recall their definition and main
properties.
\subsection{Basic definitions and results}
\label{subsec:basics_Quantspace}
\begin{dfn}
 A square matrix $\mathrm{\q}= \|q_{ij}\|$ over the complex numbers is \emph{multiplicatively anti-symmetric} if $q_{ij} \in \C^{\times}$, $q_{ji}=q_{ij}^{-1}$ and
$q_{ii}=1$ for all $i,j$.  We shall sometimes refer
to $\mathrm{\q}$ as \emph{a deformation matrix}.
\end{dfn}

\begin{dfn}
\label{dfn:quantum_space}
Let $\mathrm{\q}$ be an $(n+1)\times (n+1)$
\mas
matrix.
We denote by 
$\cA^n_\q$ the complex quadratic algebra
with $n+1$ generators $x_0, \dots, x_n$ subject to the  $\binom{n+1}{2}$
quadratic binomial relations
\begin{equation}
\label{eq:GB_ProjSp}
 \cR=\cR_\q := \{x_jx_i - q_{ji} x_ix_j\mid 0\le i<j\le n\}.
 \end{equation}
In other words
$\cA^n_\q = \C\asX / (\cR)$.
We refer to $\cA^n_\q$ as the \emph{quantum space defined
by the \mas matrix $\q$}.
\end{dfn}

Clearly, the algebra $\cA^n_\q$ is commutative if and only if all entries of $\q$
are 1. In this case $\cA^n_\q$ is isomorphic the algebra of commutative
polynomials $\C[x_0, \dotsc, x_n]$.
Although $\cA^n_{\q}$ is non-commutative whenever $\q$ has at least one entry different from 1, it preserves
 all `good properties' of the commutative polynomial ring $\C[x_0, \dotsc, x_n]$, see Facts \ref{cor:skewpol}.

\begin{ex}
For $n=2$ and
\[ \q= \begin{pmatrix}
1 & q^{-2} & 1\\
q^2 & 1 & 1\\
1& 1& 1
\end{pmatrix}\]
one obtains the non-commutative variety $\mathbb{P}^2_{q, \hslash=0}$ defined in \cite[Section 3.7]{KaKuOr01}. The quantum space $\cA^2_{\q}$ is an Artin--Schelter regular algebra of global
dimension
three, see \cite{AS}.
\end{ex}

\begin{rmk}
\label{rmk:Grbasis} It is easy to prove that the set $\cR$ defined in \eqref{eq:GB_ProjSp} is a reduced
Gr\"obner basis for the ideal $I = (\cR)$ and this fact is well known, see for example \cite[Proposition 5.5]{KRW90}.
 Therefore
 \[N(I) = N(\cR)= \cT(X_n). \]
 In other words the set $\cT(X_n)$ of ordered monomials is the normal basis of the $\C$-vector space $\cA^n_{\q}$.
The free monoid $\asX$ splits as a disjoint union
\begin{equation}
\label{eq:X1eq2}
\asX=  \cT(X_n) \sqcup \LM(I),
\end{equation}
and $\C \asX \simeq  \Span_{\C} \cT(X_n)\oplus I$.
\end{rmk}

\begin{rmk}
\label{rmk:Normalform}
\begin{enumerate}
\item
Every element $f\in \C\asX\setminus I$ has unique normal form $\Nor(f)=\Nor_{\cR}(f)= \Nor_{I} (f)$,
which
satisfies
\[
\Nor(f) = \sum_{i=1}^s c_i T_i \in \C \cT(X_n), \] where
$c_i \in  \C^{\times}$,  $T_1 <T_2 <\ldots < T_s \leq \LM(f)$, and
the equality $f = \Nor (f)$ holds in the algebra $\cA^n_{\q}$.
Moreover, $\Nor (f)= 0$ if and only if $f\in I$.
\item The normal form $\Nor(f)$ can be found effectively using a finite sequence of reductions defined via $\cR$.
   \item There is an equality $\Nor_{\cR} (x_jx_i)= q_{ji}x_ix_j$, for every $0\le i<j\le n$.
\end{enumerate}
 When the ideal $I$, or its generating set $\cR$ is understood from the context, we shall
denote the normal form of $f$ by  $\Nor (f)$.
    \end{rmk}

 More generally, recall that a quadratic  algebra is an associative graded algebra
 $A=\bigoplus_{i\ge 0}A_i$ over a ground field
 $\textbf{k}$  determined by a vector space of generators $V = A_1$ and a
subspace of homogeneous quadratic relations $R= R(A) \subset V
\otimes V.$ We assume that $A$ is finitely generated, so $\dim A_1 <
\infty$. Thus $ A=T(V)/( R)$ inherits its grading from the tensor
algebra $T(V)$. The Koszul dual algebra of $A$, denoted by $A^{!}$
is the quadratic algebra $T(V^{*})/( R^{\bot})$, see
\cite{Ma87, Ma88}. The algebra  $A^{!}$ is also referred to as \emph{the quadratic dual algebra to a quadratic algebra} $A$, see \cite{PoPo}, p.6.

Note that every quantum space $\cA= \cA^n_\q$ is \emph{a skew-polynomial ring with
binomial relations} in the sense of \cite{Ga94,Ga96}, and \emph{a quantum binomial algebra} in the sense of \cite{Ga12}.
Thus the next corollary follows straightforwardly from \cite[Theorem A]{GI04}, see also \cite{Ga12}, Lemma 5.3, and Theorem 1.1.
\begin{cor}
    \label{cor:koszuldual}
    Let $\cA= \cA^n_\q$ be a \emph{quantum space defined
by the \mas matrix $\q$}.
Then
\begin{enumerate}
\item
The Koszul dual $\cA^{!}$ has a presentation
$\cA^{!}= \C\langle \xi_0, \xi_1, \cdots, \xi_n \rangle/ (\cR^{\bot})$, where $\cR^{\bot}$ consists of
$\binom{n+1}{2}$
quadratic binomial relations and $n+1$ monomials
\begin{equation}
\label{eq:GB_ProjSp}
 \cR^{\bot}= \{\xi_j\xi_i - q_{ji}^{-1} \xi_i\xi_j\mid 0\le i<j\le n\} \cup  \{\xi_j^2\mid 0\le j\le n\}.
 \end{equation}
\item
The set $\cR^{\bot}$ is a Gr\"{o}bner basis of the ideal $(\cR^{\bot})$ in $\C\langle \xi_0, \xi_1, \cdots, \xi_n \rangle$,
so $\cA^{!}$ is a \emph{PBW} algebra with \emph{PBW}
generators $\xi_0, \xi_1, \cdots, \xi_n $.
\item
$\cA^{!}$ is a quantum Grassmann algebra of dimension $n+1$.
\end{enumerate}
\end{cor}

The following result can be extracted from \cite{Ga96,GIVB98},  and \cite[Theorem 1.1]{Ga12}.
We use the well-known equality $\binom{n+d}{n} = \binom{n+d}{d}$.

\begin{facts}
\label{cor:skewpol}
Let  $\cA =\cA^n_\q$ be a quantum space.
\begin{enumerate}
\item $\cA$ is canonically graded by length, it is generated in degree one, and $\cA_0=\C$.
\item $\cA$ is a \emph{PBW}-algebra in the sense of Priddy \cite[Section 5]{priddy}, with a \emph{PBW} basis $\cT(X_n)$.
For every $d\in\N$ there is an isomorphism of vector spaces $\cA_d \simeq \Span_{\C} \cT(X_n)_d$, so
\begin{equation}
\label{A-graded}
\dim \cA_d =\vert \cT(X_n)_d \vert = \binom{n+d}{d}.
\end{equation}
\item $\cA$ is Koszul.
\item $\cA$ is a left and a right Noetherian domain.
\item $\cA$ is an Artin--Schelter regular algebra, that is
\begin{enumerate}
\item $\cA$ has polynomial growth of degree $n+1$ (equivalently,  $\gkdim \cA=n+1$);
\item $\cA$ has finite global dimension $\gldim \cA=n+1$;
\item $\cA$ is Gorenstein.
\end{enumerate}
\item The Hilbert series of $\cA$ is
$H_{\cA}(t)= 1/(1-t)^{n+1}$.
\end{enumerate}
\end{facts}

\begin{rmk}
    \label{rmk:strictly_ordered}
The algebra  $\cA= \cA^n_\q$ is a \emph{quantum projective space} in the sense of \cite[Definition 2.1]{ShTi01}
and it is \emph{solvable} in the sense of Kandri-Rodi and Weispfenning
 \cite[Section 1]{KRW90}.
\end{rmk}

Suppose a monomial $u \in \langle X_n\rangle$ has multi-degree $\deg (u) = \alpha =(\alpha_0, \alpha_1,\dotsc, \alpha_n)$ and let  $T_\alpha= x_0^{\alpha
_0}x_1^{\alpha _1}\dotsm x_n^{\alpha _n}$, as in Definition \ref{def:multidegreedef}.
Since all relations in $\cR$ are binomials which preserve the multigrading, there exists a unique $\zeta_u \in \C^{\times}$ such that
\begin{enumerate}
    \item $\zeta_u$ is a monomial in the entries of $\q$,
    \item $\Nor_\cR (u) = \zeta_u T_{\alpha}$,
    \item $u \equiv \zeta_u T_{\alpha}  \mbox{ modulo } I$, i.e., the equality $ u = \zeta_u T_{\alpha}$ holds in $\cA^n_{\q}$.
\end{enumerate}

\begin{convention}
\label{rmk:conventionpreliminary}
    Following \cite{Bergman} (see also our Lemma   \ref{lem:diamonlemma}),
 we consider the space $\C \cT^n$
endowed with multiplication defined by
 \[f \bullet g := \Nor_\cR(fg),\]
 for every $f,g \in \C \cT^n$.
Then $(\C \cT^n, \bullet )$ has a well-defined structure of a graded algebra, and there is an isomorphism of graded algebras
\begin{equation}
\label{eq:Aas_vectspace}
\cA^n_{\q}\cong (\C \cT^n, \bullet ).
\end{equation}
By convention we shall identify the algebra $\cA^n_{\q}$ with $(\C \cT^n,\bullet )$.
\end{convention}

\subsection{Some basics of Gr\"obner bases theory for ideals in quantum spaces}
\label{subsec:grobner_Quantspace}
In Sections \ref{sec:Veronese} and \ref{Sec:Segremap} we shall introduce analogues of the Veronese map $v_{n,d}$
and of the Segre map $s_{n,m}$ for quantum spaces. A
natural problem in this context is to describe the reduced
Gr\"{o}bner bases of  $\ker (v_{n,d})$ and $\ker
(s_{n,m})$. Each of the kernels is an ideal of an
appropriate quantum space $\cA^N_\g$,
so we need a Gr\"{o}bner bases theory which is admissible for quantum spaces.
Proposition \ref{pro:strictlyordered} shows that each
quantum space $\cA^N_\g$ is a \emph{strictly ordered} algebra in the sense of
\cite[Definition 1.9]{GI91},
and the Gr\"{o}bner bases theory for
ideals in strictly ordered algebras presented by the third author in \cite{GI91} and
\cite{GI91A} seems natural and convenient for our quantum spaces.
Here we follow the approach of these works. Note that the
results of \cite{GI91} and \cite{GI91A} are independent from and
agree with \cite{KRW90} and \cite{Mo94}.

In the sequel we often work simultaneously with
two distinct quantum spaces whose sets of generators $X_n = \{x_0,
\dots, x_n\}$ and $Y_N = \{y_0, \dots, y_N\}$ are disjoint and
have different cardinalities,  $N> n$. To avoid ambiguity we
denote by $\prec$ the degree-lexicographic ordering on $\asY$ and
by $\prec_0$  the restriction $\prec_{|\cT(Y_N)}$ of $\prec$ on
the set of ordered monomials $\cT(Y_N) \subset \asY$.

Given an
arbitrary   \mas $(N+1)\times (N+1)$ matrix $\g =\|g_{ij}\|$, let $\cA^N_\g= \C\asY / (\cR_{\g})$ be the associated quantum space, where
\[\cR _{\g} := \{y_j y_i - g_{ji} y_i y_j     \mid 0\le i<j\le N \}.\]
Following Convention \ref{rmk:conventionpreliminary}, we identify the two algebras
\[\cA^N_\g \cong (\C \cT(Y_N), \bullet).\]



 Let $\mathfrak{J}_\g = (\cR_\g)$.
 We shall write $\Nor (f)$ for the normal form 
 of $f\in \C \asY$,
 keeping the ideal $\mathfrak{J}_\g $
fixed.
The operation $\bullet$ on $\C \cT(Y_N)$ induces
also an operation $\star$ on the set $\cT(Y_N)$  defined by
\[u\star v := \LM(\Nor (uv)) = \LM(u\bullet v),\]
for every $u, v \in \cT(Y_N)$.
It is not difficult to see that $(\cT(Y_N),\star)$ is a monoid.

Let $u, v \in \cT(Y_N)$, and let $\alpha=\deg u+\deg v$. We know that
$u\bullet v = \zeta(u,v) T(u,v)$, where $\zeta = \zeta(u,v) \in \C^{\times}$ and $T(u,v)\in \cT(Y_n)$, with $\deg T(u,v) = \alpha$.
Similarly, $v\bullet u = \eta(v,u) T(v,u),$ where $\eta(v,u) \in \C^{\times}$ and $\deg T(v,u) = \alpha = \deg T(u,v)$.
The unique ordered monomial in $\asY$ with multi-degree $\alpha$ is $T_{\alpha}$, therefore
 \[u\star v = v\star u = T_{\alpha}. \]
 It follows that there is an isomorphism of monoids $(\cT(Y_n), \star) \cong  [y_0, \dots, y_N]$, the free
 abelian
   monoid generated by $Y_N$.
 This agrees with  \cite[Theorem I and Theorem II]{GI91}.

Note that identifying $\cA^N_\g$ with  $\C \cT(Y_N)$ we also have
the degree-lexicographic well-ordering $\prec_0$ on
 the free abelian monoid $(\cT(Y_N), \star)$. For every $f \in \C \cT(Y_N)$, its leading monomial with respect to $\prec_0$ is denoted by $\LM (f)_{\prec_0}$.
  In fact $\LM_{\prec_0}(f) =\LM_{\prec}(f)$ and we shall simply write  $\LM (f)$.

The proposition below follows straightforwardly from \cite{GI91}.

 \begin{pro}
 \label{pro:strictlyordered}
\begin{enumerate}
\item
  The quantum space $\cA^N_\g = (\C \cT(Y_N), \bullet)$  is a  \emph{strictly ordered algebra} in the sense of \cite[Definition 1.9]{GI91},
 that is,  each of the following two equivalent conditions is
 satisfied:
\begin{description}
\item[SO1] Let $a,b,c \in \cT (Y_N)$. If $a \prec_0 b$, then $a\star c \prec_0  b\star c$ and $ c\star a \prec_0  c\star b$;
\item[SO2] $\LM (f\bullet h) = \LM (\LM(f)\bullet \LM(h))$, for all $ f,h \in
\cA^N_\g$.
\end{description}
\item
Every two-sided (respectively,  one-sided)  ideal $\mathfrak{K}$ of $\cA^N_\g$ has a finite reduced Gr\"{o}bner basis with respect to the ordering
$\prec_0$  on $(\cT(Y_N), \star)$, see Definition \ref{dfn:Grbasis_quantum_space}.
\end{enumerate}
 \end{pro}

The properties \textbf{SO1} and \textbf{SO2} allow to define  Gr\"{o}bner bases for ideals of a quantum space
$\cA^N_\g$ in a
natural way, and to use a standard Gr\"{o}bner bases
theory, analogous to the theory of non-commutative Gr\"{o}bner
bases for ideals of the free associative algebra (Diamond Lemma)
proposed by Bergman.

\begin{dfn}
\label{dfn:normal_terms}
 Let $P \subset \cA^N_\g$ be an arbitrary subset, and let $\LM(P)= \{\LM(f) \mid f \in P\}$.
 A monomial $T\in \cT (Y_N)$ \emph{is normal modulo} $P$  if it does not contain as a subword  any $u \in \LM(P)$.
 We denote
 \begin{equation}
\label{eq:N(P)}
N_{\prec_0}(P)= \{T \in \cT(Y_N) \mid T\text{ is normal mod P} \}.
\end{equation}
    \end{dfn}

\begin{dfn}
\label{dfn:Grbasis_quantum_space} Suppose $\mathfrak{K}$ is an
ideal of $\cA^N_\g = \C\cT(Y_N)$.  A set $F \subset \mathfrak{K}$
is \emph{a Gr\"{o}bner basis} of $\mathfrak{K}$ if for
any $h \in \mathfrak{K}$ there exists an $f \in F$, and monomials
$a,b \in \cT(Y_N)$ such that
   $\LM(h) = a\star\LM(f)\star b$.
   Due to the commutativity of the operation $\star$ this is equivalent to $\LM(h) = u\star\LM(f)$, for some $u \in \cT$.
    \end{dfn}
  An interested reader can find various equivalent definitions of a Gr\"{o}bner basis in \cite{Mo94,KRW90}, and numerous papers which appeared later.
  Given
  an ideal $\mathfrak{K}$
  generated by a finite set $F$ one can verify algorithmically whether $F$ is a Gr\"obner basis for the ideal
  $\mathfrak{K}$,
    see for example \cite{Mo94}.

   \begin{lem}
    \label{lem:quadr_Grbasis}
Let $\mathfrak{K} = (F)$ be an ideal of $\cA^N_\g$ generated by
the set $F \subset \C \cT(Y_N)$. Then
$F$ is a Gr\"{o}bner basis of $\mathfrak{K}$ if and only if $N(F) = N_{\prec_0}(F) = N_{\prec_0}(\mathfrak{K})$.
In this case the vector space $\cA^N_\g$ splits as a direct sum
\[\cA^N_\g = \C\cT(Y_N) = \mathfrak{K} \oplus \C N_{\prec_0}(F)\]
and the set $N_{\prec_0}(F) \subset \cT(Y_N)$ 
 projects to a $\C$-basis of the quotient algebra
$\cA^N_\g / \mathfrak{K}.$
Moreover, if $F$ consists of
homogeneous polynomials, then
\begin{equation}
\label{eq:direct_sums}
(\cA^N_\g)_j = (\C\cT(Y_N))_j = (\mathfrak{K})_j \oplus (\C N_{\prec_0}(F))_j,
\end{equation}
for every $j \geq 2$.
\end{lem}

The following is an analogue of Lemma \ref{lem:diamondquadratic} for ideals of $\cA^N_\g$ generated by quadratic polynomials.
\begin{lem}
\label{lem:quant_diamondquadratic} Let $\mathfrak{K} = (F)$ be an
ideal of $\cA^N_\g$ generated by a set of quadratic polynomials $F
\subset (\C \cT(Y_N))_2$ and let $B = \cA^N_\g / \mathfrak{K}$.
We consider the canonical grading of $B$ induced by the grading of
$\cA^N_\g$. Then $F$  is a Gr\"{o}bner basis of
$\mathfrak{K}$ if and only if
\begin{equation}
\label{eq:dim_B_3}
 \dim B_3 = |(N_{\prec_0}(F))_3|.
\end{equation}
\end{lem}

\section{The $d$-Veronese subalgebra of $\cA^n_{\q}$, its generators and relations}
\label{sec:dVeronese}
In this section we study the $d$-Veronese subalgebra $\cA^{(d)}$ of the  quantum space
$\cA=\cA^n_{\q}$. This is an
algebraic construction which mirrors the Veronese embedding.  First we recall some basic definitions and facts about Veronese subalgebras of general graded algebras.
 Our main reference is  \cite[Section 3.2]{PoPo}.
 The main result of the section is Theorem \ref{thm:d-Veronese_relations} which
 presents the $d$-Veronese subalgebra $\cA^{(d)}$ in terms of generators and explicit quadratic relations.

\begin{dfn}
Let $A=\bigoplus_{k\in\N_0}A_{k}$ be a graded algebra. For $d\in\N$, the \emph{$d$-Veronese subalgebra} of $A$ is the graded algebra
    \[A^{(d)}=\bigoplus_{k\in\N_0} A_{kd}.\]
\end{dfn}

\begin{rmk}
    \label{rmk:Ver_properties}
\begin{enumerate}
\item By definition the algebra $A^{(d)}$ is a subalgebra of $A$. However, the embedding is not a graded algebra morphism.
The Hilbert function of $A^{(d)}$ satisfies
    \[h_{A^{(d)}}(t)=\dim(A^{(d)})_t=\dim(A_{td})=h_A(td).\]
\item Let $\cA =\cA^n_\q$ be the quadratic algebra with relations
$\cR$ introduced in Definition \ref{dfn:quantum_space}. It follows
from \cite[Proposition 2.2]{PoPo}, and Facts  \ref{cor:skewpol}
that its $d$-Veronese subalgebra $\cA^{(d)}$ is one-generated, quadratic and
Koszul. Moreover, $\cA^{(d)}$ is
left and right Noetherian.
\end{enumerate}
\end{rmk}

We fix a \mas matrix $\q$ and set $\cA =\cA^n_{\q}$. By Convention
\ref{rmk:conventionpreliminary}, $\cA$ is identified with the
algebra $(\C \cT^n, \bullet)$ and
\[\cA =\bigoplus_{k\in\N_0}  \cA_k \cong \bigoplus_{k\in\N_0}  \C(\cT^n)_k.\]
Hence its d-Veronese subalgebra satisfies
\[\cA^{(d)} = \bigoplus_{k\in\N_0}  \cA_{kd} \cong \bigoplus_{k\in\N_0}  \C(\cT^n)_{kd}.  \]
The ordered monomials $w \in (\cT^n)_d$ of length $d$ are degree one generators of $\cA^{(d)}$, hence
\[
\dim \cA_d = |(\cT^n)_d |=\binom{n+d}{d}.
\]
We set $N=\binom{n+d}{d}-1$ and we order the elements of $(\cT^n)_d$  lexicographically, so
        \begin{equation}
        \label{eq:deg-d-monomials}
(\cT^n)_d = \{ w_0 =x_0^d < w_1= (x_0)^{d-1}x_1 < \dots < w_N= x_n^d \}.
        \end{equation}
The $d$-Veronese $\cA^{(d)}$ is a quadratic algebra
(one)-generated by $w_0, w_1, \dots , w_N.$ We shall find a
minimal set of its quadratic relations, each of which is a linear
combination of products  $w_iw_j$ for some $i,j\in\{0,\dots,N\}$. The
following notation will be used throughout the paper.

\begin{notaz}
\label{notaz: C(n,d)} 
 Let $N=\binom{n+d}{d}-1$.  For every integer $j, \; 1 \leq j \leq N$, we denote by $\alpha^{j}$ the  multi-degree $\deg
(w_j)$, thus
 \[\alpha^j=(\alpha _{j_0},  \dots, \alpha _{j_n}) \mbox{ whenever }  w_j = x_0^{\alpha_{j_0}} \dots x_n^{\alpha_{j_n}}.\]
We define
\begin{equation}
m(j) = \min \{s \in \{0, \dots, n\}\mid \alpha_{j_s} \geq 1
\}\mbox{ and }  M(j) = \max \{s \in \{0, \dots, n\}\mid
\alpha_{j_s} \geq 1 \}.
\end{equation}
In other words, if $ w_j = x_{j_1}^{\alpha_{j_1}}x_{j_2}^{\alpha_{j_2}}\dots x_{j_d}^{\alpha_{j_d}}$ for some $0 \leq j_1 \leq j_2 \leq \dots \leq j_d$ and ${\alpha_{j_1}},
\dots,
{\alpha_{j_d}} \geq 1$, then $m(j)= j_1$ and $M(j) = j_d$. For example, if  $w_j = x_2x_4^3x_7^2$, then $m(j) = 2$ and $M(j) = 7$.
We further define
\begin{align*}
\mathrm{P}(n,d)    &= \{ (i, j)\mid 0 \leq i \leq j \leq N \};  \\
\mathrm{C}(n,2, d) &= \lbrace (i, j)\in\mathrm{P}(n,d)\mid M(i) \leq m(j)\rbrace=\lbrace (i, j)\in \mathrm{P}(n,d)\mid w_iw_j\in(\cT^n)_{2d}\rbrace;\\
\mathrm{C}(n,3,d) & = \lbrace (i, j, k) \, \vert \, 0 \leq i \leq j \leq k \leq N, \, (i,j), (j, k) \in \mathrm{C}(n,2, d)   \rbrace; \\
\Mv(n,d) & = \lbrace (i, j)\in\mathrm{P}(n,d)\mid M(i) > m(j)\rbrace=\lbrace (i, j)\in\mathrm{P}(n,d)\mid w_iw_j\notin(\cT^n)_{2d}\rbrace.
\end{align*}
\end{notaz}

\begin{lem}
    \label{lem:orders}
Let $(\cT^n)_{p} = (\cT(X_n))_{p}$ be the set of all ordered monomials $w \in \asX$ of length  $|w|= p$.
  \begin{enumerate}
  \item
  The maps
  \begin{align*}
  \Phi: \mathrm{C}(n,2, d) &\rightarrow (\cT^n)_{2d}&\mbox{ and }&&\Psi: \mathrm{C}(n,3, d) &\rightarrow (\cT^n)_{3d}\\ (i,j)& \mapsto w_iw_j &&&(i,j,k)& \mapsto w_iw_jw_k
  \end{align*}
  are bijective. Therefore
  \begin{equation}
  \label{eq:card_Cnd}
 |\mathrm{C}(n,2, d) | = |(\cT^n)_{2d}| =\binom{n+2d}{n}\quad \mbox{ and }\quad
 |\mathrm{C}(n,3, d) | = |(\cT^n)_{3d}| =\binom{n+3d}{n}.
\end{equation}
  \item
  The set $\mathrm{P}(n, d)$ is a disjoint union
$\mathrm{P}(n, d) = \mathrm{C}(n,2, d) \sqcup \Mv(n,d).$ Moreover
\begin{equation}
  \label{eq:card_Mv}
|\mathrm{P}(n, d)| = \binom{N+2}{2}\quad \mbox{ and }\quad
 |\Mv(n,d)|= \binom{N+2}{2} - \binom{n+2d}{n}.
\end{equation}
  \end{enumerate}
\end{lem}

\begin{proof} \begin{enumerate}
\item
Given $w_i, w_j \in (\cT^n)_{d}$, their product $w = w_iw_j$ belongs to $(\cT^n)_{2d}$  if and only if
$(i,j)\in  \mathrm{C}(n,2, d)$, hence $\Phi$ is well-defined.
Observe that every $w \in (\cT^n)_{2d}$ can be written uniquely as
\begin{equation}
\label{eq:w_Tn_2d} w = x_{i_1}\dots
x_{i_d}x_{j_{1}}\dots x_{j_{d}}, \mbox{ where }   0 \leq i_1
\leq \dots \leq i_d \leq j_1  \leq \dots \leq
j_d.
\end{equation}
It follows that $w$ has a unique presentation $w=w_iw_j$, 
 where
\begin{align*}
 w_i&= x_{i_1}\dots x_{i_d} \in (\cT^n)_{d}\mbox{, }
 w_j =x_{j_{1}}\dots x_{j_{d}}\in (\cT^n)_{d},\\
  M(i) &=i_d\leq  m(j)= j_1\mbox{ and }
 (i,j)\in  \mathrm{C}(n,2, d).
\end{align*}
This implies that $\Phi$ is a
bijection
.

Consider now the map $\Psi$. Given $w_i,
w_j, w_k \in (\cT^n)_{d}$, their product $\omega = w_iw_jw_k$
(considered as an element in $\asX$) belongs to $(\cT^n)_{3d}$ if and only if $(i,j,k)\in
\mathrm{C}(n,3, d)$, hence  $\Psi$ is well-defined. The proof that $\Psi$ is bijective is similar to the case of $\Phi$.
\item It is clear that
\[
   |\mathrm{P}(n, d)| =\binom{N+1}{2}+ N+1 = 
    \binom{N+2}{2}.
   \]
By definition $\mathrm{P}(n,d)=\mathrm{C}(n,2, d) \sqcup
\Mv(n,d)$  is a disjoint union of sets, hence
\[
|\Mv(n,d)| = |\mathrm{P}(n, d)| -|\mathrm{C}(n,2,d)| = \binom{N+2}{2} -\binom{n+2d}{n}.\qedhere
\]\end{enumerate}
 \end{proof}


The following result describes the $d$-Veronese subalgebra $(\cA^n_\q)^{(d)}$ of the quantum space $\cA^n_\q$ in terms of generators and quadratic relations.

\begin{thm}
    \label{thm:d-Veronese_relations} Let $\q$ be an $(n+1)\times(n+1)$ \mas matrix and let $\cA =\cA^n_\q$.
The d-Veronese subalgebra $\cA^{(d)} \subseteq \cA$
is a quadratic algebra with $\binom{n+d}{d}$ generators, namely the elements of $(\cT^n)_{d}$,
subject to $(N+1)^2 -\binom{n+2d}{n}$ independent quadratic  relations
 which split into two disjoint sets $\cR_1$ and $\cR_2$ given below.
\begin{enumerate}
 \item
 \label{thm:d-Veronese_relations1}
 The set $\cR_1$ contains exactly $\binom{N+1}{2}$  relations
 \begin{equation}
  \label{eq:fji}
\cR_1 = \{ f_{ji} = w_jw_i -\varphi_{ji}
w_{i^{\prime}}w_{j^{\prime}}\mid 0 \leq i < j \leq N, \;
(i^{\prime}, j^{\prime})\in \mathrm{C}(n,2,d),\;
      \varphi_{ji}
\in \C^{\times}\},
\end{equation}
where for each pair $j >i$ the product $w_jw_i$ occurs exactly once in $\cR_1$, and
there is unique pair $(i^{\prime}, j^{\prime})\in \mathrm{C}(n,2,d)$ such that
 $\Nor(w_jw_i) = \varphi_{ji}  w_{i^{\prime}}w_{j^{\prime}}=  \varphi_{ji}T_{\beta}$,
 with $\beta =\deg(w_jw_i) = \deg(w_{i^{\prime}} w_{j^{\prime}})$.
 One has
 \[\LM(f_{ji}) = w_jw_i > w_{i^{\prime}}w_{j^{\prime}}= T_{\beta}\in (\cT^n)_{2d}.\]
Moreover, for every pair $(i,j) \in \mathrm{C}(n,2,d)$ such that $i < j$, the product $w_iw_j = T_{\beta} \in (\cT^n)_{2d}$ occurs in a relation $w_jw_i -
\varphi_{ji}w_iw_j \in \cR_1$.
Each coefficient $\varphi_{ji}$ is a non-zero complex number, uniquely determined by $\q$.

\item
\label{thm:d-Veronese_relations2}
The set $\cR_2$
consists of exactly   $\binom{N+2}{2} - \binom{n+2d}{n}$ relations
\begin{equation}
  \label{eq:fij}
\cR_2 = \{ f_{ij} = w_i w_j - \varphi_{ij}
w_{i^{\prime}}w_{j^{\prime}}\mid (i,j) \in   \Mv (n,d), \;
(i^{\prime}, j^{\prime})\in \mathrm{C}(n,2,d), \;
\varphi_{ij}  \in \C^{\times}\},
\end{equation}
where for each pair  $(i,j) \in   \Mv (n,d)$ the word  $w_i w_j$
occurs exactly once in $\cR_2$, and determines uniquely a pair
$(i^{\prime},j^{\prime})\in \mathrm{C}(n,2,d)$ with
$i^{\prime}<j^{\prime}$, and a nonzero complex number
$\varphi_{ij}$ such that $\Nor(w_iw_j) = \varphi_{ij}
w_{i^{\prime}}w_{j^{\prime}} = \varphi_{ij}T_{\beta}$, with
$\beta=\deg(w_iw_j) = \deg(w_{i^{\prime}} w_{j^{\prime}})$.
 In particular,
 \[\LM(f_{ij}) = w_iw_j >  w_{i^{\prime}}w_{j^{\prime}}= T_{\beta}\in (\cT^n)_{2d}.\]

\item
\label{thm:d-Veronese_relations3}
The relations $\cR_1 \cup \cR_2$ imply a set $\cR_1^{\prime}$ of  $\binom{N+1}{2}$ additional relations:
 \begin{equation}
  \label{eq:gij}
\cR_1^{\prime} = \{w_j\cdot w_i -
g_{ji}  w_i\cdot w_j\mid g_{ji} \in \C^{\times},      0 \leq i < j \leq N\},
\end{equation}
where for each $i < j$ the coefficient  $g_{ji} = \frac{\varphi_{ji}}{\varphi_{ij}}$ is uniquely determined by the matrix $\q$.
We set $\varphi_{ij} = 1$ whenever $(i,j) \in \mathrm{C}(n,2,d)$.
\item
\label{thm:d-Veronese_relations4}
Conversely, the relations $\cR^{\prime}= \cR_1^{\prime} \cup\cR_2$ imply the relations $\cR_1$. Moreover, $\cR^{\prime}$ is also a complete set of independent relations for
the $d$-Veronese algebra $\cA^{(d)}$.
\end{enumerate}
    \end{thm}

\begin{proof}
\begin{enumerate}
    \item Suppose that $0\leq i<j\leq N$. Then $w_j > w_i$, and it is not difficult to see that $M(j) > m(i)$, so $w_jw_i$ is not in normal form.
By Remark \ref{rmk:Normalform}, its normal form has the shape
$\Nor (w_jw_i) = \varphi_{ji}  T_{\beta},$ where $\beta = \deg(w_jw_i) = \alpha^i + \alpha^j$, and $\varphi_{ji} \in \C^\mathfrak{\times}$ is uniquely determined by the
entries of $\q$. By Lemma \ref{lem:orders},
$T_{\beta} = w_{i^{\prime}} w_{j^{\prime}}$ for a unique pair $(i^{\prime}, j^{\prime})\in \mathrm{C}(n,2,d)$
of ordered monomials $w_{i^{\prime}}\leq w_{j^{\prime}}$ of length $d$.
We claim that $w_{i^{\prime}}<w_{j^{\prime}}$.

Assume by contradiction that $w_{i^{\prime}}=w_{j^{\prime}} =
x_{i_1}x_{i_2} \dots x_{i_d}$, where $x_{i_1}\leq x_{i_2} \leq
\dots \leq x_{i_d}$. This implies that
$w_{i^{\prime}}w_{j^{\prime}} = w_{i^{\prime}}^2 = x_{i_1}x_{i_2}
\dots x_{i_d} x_{i_1}x_{i_2} \dots x_{i_d} \in (\cT^n)_{2d}$. But
this is possible if and only if $x_{i_k} = x_{i_1}$ for every
$k\in\{2,\dots,d\}$, that is $w_{i^{\prime}}= w_{j^{\prime}}=
(x_p)^d,$ for some $p\in\{0,\dots,n\}$, so $T_{\beta}
=(x_p)^{2d}$. In other words $\beta = (\beta_0, \dots, \beta_n)$,
where $\beta_p = 2d$ and $\beta_i =0$ for every $i \neq p$. One
has $\beta= \deg (w_jw_i) = \deg(w_j)+ \deg(w_i)= \alpha ^j
+\alpha^i,$ which together with $|w_i|=|w_j|= d$ imply $\alpha ^i
=\alpha^j$ and $w_i=w_j = (x_p)^d$, which is impossible, since by
assumption $i<j$. Hence $w_{i^{\prime}}< w_{j^{\prime}}$  and
$i^{\prime}< j^{\prime}$. We know that the equality $w_jw_i=
\Nor(w_jw_i)$ holds in $\cA$, hence it is an equality in
$\cA^{(d)}$. This implies that the equality $(w_jw_i) =
\varphi_{ji} w_{i^{\prime}} w_{j^{\prime}}$ holds in $\cA^{(d)}$,
for all $0 \leq i < j \leq N$. It follows that $\cA^{(d)}$
satisfies the relations $f_{ji}= 0$, for all $f_{ji} \in \cR_1$,
see (\ref{eq:fji}). Moreover, the relations satisfy the properties
given in part (\ref{thm:d-Veronese_relations1}). It is clear that
the order of $\cR_1$  is exactly  $\binom{N+1}{2}$.

\item
Suppose that $(i,j) \in \Mv(n,d)$. Then the following are equalities in $\cA$:
\[
w_i w_j = \Nor (w_iw_j) =\varphi_{ij}  T_{\beta},  \mbox{ where } T_{\beta} < w_iw_j,
\beta = \alpha^i + \alpha^j,
\]
and $\varphi_{ij} \in \C^{\times}$ is uniquely determined by the entries of $\q$. By Lemma
\ref{lem:orders}, $T_{\beta} = w_{i^{\prime}} w_{j^{\prime}}$ for a unique pair $(i^{\prime}, j^{\prime})\in \mathrm{C}(n,2,d)$.
We claim that $w_{i^{\prime}}<w_{j^{\prime}}$. As in part (1), assuming that $w_{i^{\prime}}= w_{j^{\prime}}$ we obtain that
$w_i=w_j = (x_p)^d$, but then $w_iw_j =(x_p^d)(x_p^d) \in (\cT^n)$, which contradicts our assumption $(i,j) \in \Mv(n,d)$.
The equality $w_i w_j = \Nor (w_iw_j)$ holds in $\cA$,  therefore it is an equality in $\cA^{(d)}$. We have shown that for every pair $(i,j) \in \Mv(n,d)$ there is unique pair  $(i^{\prime}, j^{\prime})\in \mathrm{C}(n,2,d)$
such that $i^{\prime}< j^{\prime}$ and $w_jw_i= \varphi_{ii} w_{i^{\prime}} w_{j^{\prime}}$ holds in $\cA^{(d)}$. Therefore $\cA^{(d)}$ satisfies the relations (\ref{eq:fij}) from $\cR_2$.
It is clear that all properties listed in part (\ref{thm:d-Veronese_relations2}) hold and $|\cR_2| = |\Mv(n,d)| = \binom{N+2}{2} - \binom{n+2d}{n}.$
Note that
\[\LM(\cR_1) = \{w_jw_i \mid  w_j > w_i \}\mbox{ and }\LM(\cR_2) = \{w_iw_j \mid w_i \leq w_j, (i,j) \in \Mv(n,d) \}.\]
It follows that
$\LM(\cR_1) \cap  \LM(\cR_2) = \emptyset$ and therefore $\cR_1 \cap \cR_2 = \emptyset$.
Hence the set of relations $\cR$ is a disjoint union $\cR =\cR_1 \sqcup \cR_2$ and

\begin{equation}
  \label{eq:cardinality_of_R}
  \begin{array}{ll}
|\cR| &= |\cR_1|+ |\cR_2|=\binom{N+1}{2}+\binom{N+2}{2} - \binom{n+2d}{n}\\
      &=(N+1)^2-\binom{n+2d}{n}\\
            &= \binom{n+d}{n}^2-\binom{n+2d}{n}.
\end{array}
\end{equation}
\item Assume now that $0 \leq i < j \leq N$. Two cases are possible.

(a) $(i,j) \in \mathrm{C}(n,2,d)$.
In this case $(i^{\prime}, j^{\prime}) = (i,j)$ and  $w_j w_i =
\varphi_{ji}  w_i w_j = \varphi_{ji} w_{i^{\prime}}w_{j^{\prime}}$,
so $g_{ji}= \varphi_{ji}$.

(b) $(i,j) \in \Mv(n,d)$.  Then the two relations
\[
w_j w_i =
\varphi_{ji}  w_{i^{\prime}}w_{j^{\prime}}\mbox{ and } w_i w_j =
\varphi_{ij}  w_{i^{\prime}}w_{j^{\prime}} \]
imply
\[
(\varphi_{ji})^{-1} w_j\cdot w_i = w_{i^{\prime}}w_{j^{\prime}} = (\varphi_{ij})^{-1} w_i\cdot w_j,
\]
and therefore
$w_j\cdot w_i = \frac{\varphi_{ji}}{\varphi_{ij}}w_i\cdot w_j$.
It follows that $w_jw_i = g_{ji}w_iw_j$, where the nonzero coefficient $g_{ji}= \frac{\varphi_{ji}}{\varphi_{ij}}$ is uniquely determined by $\q$.
\item This is analogous to (\ref{thm:d-Veronese_relations3}).\qedhere\end{enumerate}
 \end{proof}

Observe that Theorem \ref{thm:d-Veronese_relations} contains important numerical data about the d-Veronese $(\cA^n_\q)^{(d)}$, which will be used in the sequel,
and which we summarise below.
\begin{connotation}
\label{notaz: data}
Let $\cA^n_\q$ be the quantum space defined via a \mas $(n+1)\times(n+1)$ matrix $\q$. Let $d \geq 2$ and $N = \binom{n+d}{n}-1
$.
We associate to the $d$-Veronese $(\cA^n_\q)^{(d)}$  a  list  $D (\cA^n_\q)^{(d)}$  of invariants
uniquely determined by $\q$ and $d$.

Let $\mathfrak{F}_1= \{\varphi_{ji}\mid 0\leq i < j \leq N\}$ be the set of coefficients occurring in $\cR_1$ (see  (\ref{eq:fji})) and let $\mathfrak{F}_2= \{\varphi_{ij}\mid (i,j) \in
\Mv (n,d)\}$ be the set of coefficients occurring in $\cR_2$ (see (\ref{eq:fij})). Let $\g= \|g_{ij}\|$ be the   \mas $(N+1)\times (N+1)$ matrix whose  entries
$g_{ij}, 0 \leq i < j \leq N
$ are the coefficients occurring in $\cR_1^{\prime}$
see (\ref{eq:gij}).
 We collect this information about $(\cA^n_\q)^{(d)}$ in the following data:
 \begin{equation*}\label{eq:data2}
\begin{array}{ll}
D (\cA^n_\q)^{(d)}: \quad &  \q = \|q_{ij}\|; \\
                                   &\mathfrak{F}_1= \{\varphi_{ji} \mid 0\leq i < j \leq N\},    \text{the set of coefficients occurring in (\ref{eq:fji})}; \\
                                    & \mathfrak{F}_2= \{\varphi_{ij} \mid  i \leq j, (i,j) \in   \Mv (n,d)\},    \text{the set of coefficients occurring in (\ref{eq:fij})};\\
                                    &\g= \|g_{ij}\|, \text{ a \mas $(N+1)\times (N+1)$ matrix with}\\&\\
                                    &g_{ji}=   \begin{cases}
           1                                  & \text{for }  i=j \\
           (\varphi_{ji})/(\varphi_{ij})  & \text{for }  (i,j)\in \Mv (n,d)\mbox{ and } i<j\\
            \varphi_{ji}                      & \text{for }  (i,j) \in \mathrm{C}(n,2,d)\mbox{ and } i<j.
                         \end{cases}
\end{array}
\end{equation*}
\end{connotation}

\section{Veronese maps}
\label{sec:Veronese}
Let $n, d\in\N$ and let $N=\binom{n+d}{d}-1$. In this section, we introduce and study non-commutative analogues of the Veronese embeddings $V_{n,d}: \mathbb{P}^n \to
\mathbb{P}^N$.
The main result of the section is Theorem \ref{thm:preVeronese_ker}, which describes explicitly the reduced Gr\"obner bases for the kernel of the non-commutative
Veronese map.

We keep the notation and conventions from the previous sections, so $X_n = \{x_0, \dots, x_n\}$ and $\cT^n = \cT(X_n)\subset \asX$ is the set of ordered monomials (terms) in
the
alphabet $X_n$.
The set $(\cT^n)_d$ of all degree $d$ terms is enumerated according the degree-lexicographic order in $\asX$:
\begin{equation}
  \label{eq:Tnd}
(\cT^n)_d = \{ w_0 =x_0^d < w_1= (x_0)^{d-1}x_1 < \dots < w_N= x_n^d \}.
\end{equation}
We introduce a second set of variables $Y_N = \{y_0, \dots,
y_N\}$, and given an arbitrary \mas $(N+1)\times (N+1)$ matrix $\g
=\|g_{ij}\|$, we present the corresponding quantum space as
$\cA^N_\g= \C\asY / (\cR_{\g})$, where
\[\cR _{\g} := \{y_j y_i - g_{ij} y_i y_j     \mid 0\le i<j\le N \}.\]
\subsection{Definitions and first results}
\begin{lem}
    \label{lem:Ver_well-defined1}
Let $n, d\in\N$ and let $N=\binom{n+d}{d}-1$. Let $(\cT^n)_d$ and $\;Y_N$ be as above.
For every $(n+1)\times (n+1)$ \mas matrix $\q$, there exists a unique $(N+1)\times
(N+1)$ \mas matrix $\g =\|g_{ij}\|$ such that the assignment
    \[y_0 \mapsto w_0,  y_1 \mapsto w_1,  \dots ,  y_N  \mapsto w_N\]
extends to an algebra  homomorphism \[v_{n,d}:\cA^N_\g \rightarrow \cA^n_\q.\]
The entries of $\g$ are given explicitly in terms of the data $D((\cA^n_\q)^{(d)})$ of the d-Veronese $(\cA^n_{\q})^{(d)}$, see (\ref{eq:data2}).
The image of the map $v_{n,d}$ is the d-Veronese subalgebra $(\cA^n_{\q})^{(d)}$.
    \end{lem}
We call $v_{n,d}$ \textit{the $(n,d)$-Veronese map}.
    \begin{proof}
  Suppose $\q$ is an $(n+1)\times (n+1)$  \mas matrix, and let $\cA^n_\q$ be the corresponding quantum space.
 Assume that there exists an $(N+1)\times (N+1)$ \mas matrix $\g$ such that the map $v_{n,d}$ is a homomorphism of $\C$-algebras. Then
    \begin{align*}
        w_jw_i=v_{n,d}(y_jy_i)= v_{n,d}(g_{ji}
        y_iy_j)=g_{ji}w_iw_j,
        \end{align*}
        for every $0 \leq i \le j \leq N$.
By Theorem \ref{thm:d-Veronese_relations},
 \[w_j w_i = \varphi_{ji} T_{\beta}\mbox{ and }w_i w_j = \varphi_{ij} T_{\beta}, \]
 for every $0 \leq i<j \leq N$,
 where   $T_{\beta} \in (\cT^n)_{2d}$ is the unique ordered monomial of multi-degree $\beta = \deg(w_j) +\deg(w_i)$.
In the particular cases when $(i,j) \in \mathrm{C}(n,2,d)$, one has $w_i w_j = T_{\beta}$, so $\varphi_{ij}= 1$.
The nonzero coefficients $\varphi_{ji}$ and $\varphi_{ij}$ are uniquely determined by the matrix $\q$,  see (\ref{eq:data2}).
It follows that the
equalities
\begin{align*}
    \varphi_{ji} T_{\beta}= w_jw_i =g_{ji}w_iw_j = g_{ji}\varphi_{ij} T_{\beta}
        \end{align*}
hold in $\cA^n_\q$, so $(g_{ji}\varphi_{ij} -\varphi_{ji}) T_{\beta} = 0$. But $T_{\beta}$ is in the $\C$-basis of $\cA^n_\q$,
and therefore
        \begin{equation}
        \label{eq:gji}
g_{ji} =  \frac{\varphi_{ji}}{\varphi_{ij}}\in \C^{\times},
        \end{equation}
        for all $0 \leq i \le j \leq N$, which agrees with (\ref{eq:data2}). This determines a unique \mas matrix $\g$ with the required properties, and therefore the quantum
space $\cA^{N}_{\g}$
is also uniquely determined.
The image of $v_{n,d}$ is the subalgebra of $\cA^n_\q$ generated by the ordered monomials $\cT_d$, which by Theorem \ref{thm:d-Veronese_relations} is exactly the $d$-Veronese
$(\cA^n_\q)^{(d)}$.

Conversely, if $\g =\|g_{ij}\|$ is an $(N+1)\times (N+1)$ matrix
whose entries satisfy (\ref{eq:gji}) then $\g$ is a \mas matrix
which determines a quantum space $\cA^N_\g$  and
the Veronese map $v_{n,d}:\cA^N_\g \rightarrow \cA^n_\q$, $y_i
\mapsto w_i, 0\leq i \leq N,$ is well-defined.
    \end{proof}

We fix an $(n+1)\times (n+1)$ \mas matrix $\q$ defining the
quantum space $\cA^n_\q$. Let $\cA^N_\g$ be the quantum space
defined via the $(N+1)\times (N+1)$  matrix $\g$ from Lemma
\ref{lem:Ver_well-defined1}. To simplify notation, as in the
previous subsection, we shall write $\cA = \cA^n_\q$. We know that
there is a standard finite presentation $\cA^N_\g= \C\asY /
(\cR_{\g})$, where
\begin{equation}
 \label{eq:Rg}
 \cR _{\g} := \{y_j y_i - g_{ji} y_i y_j     \mid 0\le i<j\le N \}
 \end{equation}
is the reduced Gr\"{o}bner basis of the ideal $J=(\cR_{\g}) = \ker
\rho $, where $\rho$ is the canonical projection
\begin{equation}\label{eq:projection}\rho : \C\asY \to
\C\asY / (\cR_{\g}) =\cA^N_\g.\end{equation}

We can lift the Veronese map $v_{n,d}:\cA^N_\g \to \cA$ to a uniquely determined homomorphism $V: \C\asY \to \cA^{(d)}$
extending the assignment
\[y_0 \mapsto w_0,  y_1 \mapsto w_1,  \dots ,  y_N  \mapsto w_N.\]
It is clear that the map $V$ is surjective, since the restriction  $V_{|Y_N}: Y_N \rightarrow (\cT^n)_d$ is bijective, and
the set of ordered monomials $(\cT^n)_d$ generates $\cA^{(d)}$.

Let $K:= \ker V \subset \C\asY$.
We want to find the reduced Gr\"{o}bner basis $\cR_0$ of the ideal $K$ with respect to the degree-lexicographic order $\prec$ on $\asY$, where $y_0 \prec \dots
\prec
y_n$.

Heuristically, we use the explicit information on the d-Veronese subalgebra $\cA^{(d)}$ given in terms of generators and relations in Theorem \ref{thm:d-Veronese_relations},
\eqref{eq:fji}, and \eqref{eq:fij}.
 In each of these relations we replace $w_i$ with $y_i$, $0\leq i \leq N,$
preserving the remaining data (the coefficients and the sets of indices),
 and obtain a polynomial in $\C\asY$.
This yields two disjoint sets of linearly independent quadratic binomials $\Re_1$ and $\Re_2$ in $\C\asY$:
\begin{enumerate}
\item the set $\Re_1$, corresponding to the set $\cR_1$ defined in \eqref{eq:fji},
consists of $\binom{N+1}{2}$ quadratic relations:
\begin{equation}
  \label{eq:Fji}
\Re_1= \{F_{ji} = y_jy_i - \varphi_{ji}y_{i^{\prime}}y_{j^{\prime}}\mid 0 \leq i < j \leq N,  i^{\prime}<  j^{\prime},    (i^{\prime},   j^{\prime})\in \mathrm{C}(n,2,
d),
  y_jy_i \succ y_{i^{\prime}}y_{j^{\prime}},         \varphi_{ji} \in \C^{\times}\};
\end{equation}
\item the set $\Re_2$, corresponding to the set $\cR_2$ defined in \eqref{eq:fij},
has exactly   $\binom{N+2}{2} - \binom{n+2d}{n}$ relations:
\begin{equation}
  \label{eq:Fij}
  \Re_2= \{F_{ij} = y_i y_j - \varphi_{ij} y_{i^{\prime}}y_{j^{\prime}}\mid (i,j) \in   \Mv (n,d),  i^{\prime}< j^{\prime},  (i^{\prime}, j^{\prime})\in
\mathrm{C}(n,2,d),\\
        y_iy_j \succ y_{i^{\prime}}y_{j^{\prime}},   \varphi_{ij} \in \C^{\times}\}.
\end{equation}\end{enumerate}

There is one more set which is contained in $K$:  the set $\cR _{\g}$ 
of defining relations for $\cA^N_\g$. Note that $\cR _{\g}$
corresponds exactly to $\cR_1^{\prime}$ from (\ref{eq:gij}). We
set $\Re = \Re_1\cup\Re_2$ and $\Re^{\prime} = \cR
_{\g}\cup\Re_2$. It is not difficult to see that
there are equalities of ideals in $\C\asY$:
\[
(\Re)=(\Re_1, \Re_2) = (\Re^{\prime})=  ( \cR _{\g},  \Re_2)
\]
and that the set of relations
$\Re$ and $\Re^{\prime}$ are equivalent.

It is clear that the set $\Re = \Re_1\cup \Re_2$ of quadratic polynomials in $\C \asY$ and the set $\cR=
\cR_1\cup \cR_2$ of relations of the d-Veronese subalgebra $\cA^{(d)}$
from Theorem \ref{thm:d-Veronese_relations} have the same cardinality. In fact
\begin{equation}
  \label{eq:orderRe}
|\Re^{\prime}|= |\Re|= |\cR| = (N+1)^2 - \binom{n+2d}{n},
\end{equation}
as computed in (\ref{eq:cardinality_of_R}).
We shall prove that the set $\Re = \Re_1\cup \Re_2$ is the reduced Gr\"{o}bner basis of $K$, while $\Re^{\prime}$ is a minimal Gr\"{o}bner basis of $K$.

\begin{thm}
    \label{thm:preVeronese_ker}
With notation as above, let
$V: \C\asY \rightarrow \cA^{(d)}$ be the algebra homomorphism extending the assignment \[y_0 \mapsto w_0,  y_1 \mapsto w_1,  \dots ,  y_N  \mapsto w_N,\]
let $K$ be the kernel of $V$.
Let $\Re = \Re_1 \cup \Re_2$ be the set of quadratic polynomials given in
(\ref{eq:Fji}) and  (\ref{eq:Fij}), and  let
$\Re^{\prime} = \cR _{\g}\cup\Re_2$, where $\cR _{\g}$ is given in (\ref{eq:Rg}). Then
\begin{enumerate}
\item
$\Re$ is the reduced Gr\"{o}bner basis of the ideal $K$.
\item
$\Re^{\prime}$ is a minimal Gr\"{o}bner basis of the ideal $K$.
\end{enumerate}
    \end{thm}

\begin{proof}
We start with a general observation. The quantum space $\cA
=\cA^n_{\q}$ is a quadratic algebra, therefore its $d$-Veronese
$\cA^{(d)}\cong \C \asY/K$ is also quadratic, see Remark
\ref{rmk:Ver_properties}. Hence $K$ is generated by quadratic
polynomials and it is graded by length.

\begin{rmk}
    \label{rmk:LMremark1}
It is clear that the sets of leading monomials and the sets of normal monomials satisfy the following equalities in $\asY$:
\begin{equation}
  \label{eq:LMonomial}
  \begin{array}{l}
\LM(\cR _{\g}) = \LM(\Re_1) = \{y_jy_i \mid 0 \leq i < j \leq N\}\\
\LM(\Re_2) = \{y_iy_j \mid  (i,j) \in \Mv (n,d)\}\\
\LM (\Re) = \LM(\Re_1)\cup \LM(\Re_2) = \LM (\Re^{\prime})\\
N(\Re) = N(\Re^{\prime}).
\end{array}
\end{equation}
Therefore  $\Re^{\prime}$ is a minimal Gr\"{o}bner basis of the ideal $K$
 if and only if $\Re$ is a reduced Gr\"{o}bner basis of $K$.
\end{rmk}

By Theorem \ref{thm:d-Veronese_relations}, the quadratic polynomials $F_{ji}(Y_n)$ in (\ref{eq:Fji}) and $F_{ij}(Y_n)$ in (\ref{eq:Fij}) satisfy
\[V(F_{ji}(y_0, \dotsc, y_N)) = f_{ji}(w_0, \dotsc, w_N) = 0,\mbox{ for every }0 \leq i < j \leq N \]
and
\[V(F_{ij}(y_0, \dotsc, y_N)) = f_{ij}(w_0, \dotsc, w_N) = 0,\mbox{ for every }(i,j) \in \Mv (n,d).\]
Thus $\Re \subset K$ and, in a similar way, $\Re^{\prime}\subset K$.
We shall show that $\Re$ is a reduced Gr\"{o}bner basis of $K$.

As usual, $N(K) \subset \C\asY$ denotes the set of normal monomials  modulo $K$, and $N(\Re) \subset \C\asY$ denotes the set of normal words modulo $\Re$.
In general, \[N(K) \subseteq N(\Re),\] and by  Corollary~\ref{cor:diamonlemma} equality holds if and only if $\Re$ is a Gr\"{o}bner basis of $K$.
Recall from Subsection \ref{subs:quadratic} that there are isomorphisms of vector spaces
\[\C \asY = K \oplus \C N(K), \quad \mbox{and} \quad \C N(K) \cong \C\asY/ K \cong \cA^{(d)}.\]
The ideal $K$ is graded by length, i.e. $K = \bigoplus_{j \geq 0} K_j$, with $K_0 = K_1 = 0$.

For $j\ge 0$, let $N(K)_j$ be the set of normal words of length $j$, with the convention that $N(K)_0 = \{1\}$,  $N(K)_1 = Y_n$.
As vector spaces,
 \[(\C\asY)_j = K_j\oplus \C N(K)_j, \quad \mbox{and} \quad \C N(K)_j \cong \cA^{(d)}_j = \cA_{jd}, \mbox{ for every }j \geq 2. \]
In particular, $(\C\asY)_2 = K_2\oplus \C N(K)_2$,
 so
\[\dim (\C\asY)_2 = \dim K_2 + \dim (\C N(K)_2) =  \dim K_2 + \dim \cA_{2d}.\]
We know that
\[\dim \cA_{2d}= \binom{n +2d}{n} \quad \mbox{and} \quad \dim (\C\asY)_2 = |(Y_n)^2|= (N+1)^2,\]
 where $Y_n^2$ is the set of all words of length two in $\asY$.
This, together with  \eqref{eq:orderRe}, implies
\[\dim K_2 =(N+1)^2 -\binom{n +2d}{n} = |\Re|.\]
Clearly, the set $\Re$ consists of linearly independent polynomials, therefore $\dim K_2 = \dim \C \Re = |\Re|$. It follows that
$\C \Re = K_2,$ and since $K$ is generated by quadratic polynomials,
one has $K = (\Re)$.

We shall use the following remark.
\begin{rmk}
    \label{rmk:NRe}
The following are equivalent:
\begin{enumerate}
\item $y_iy_jy_k  \in N(\Re)_3$;
\item $y_iy_j  \in N(\Re)_2$ and $y_jy_k  \in N(\Re)_2$;
\item $(i,j,k) \in \mathrm{C}(n,3, d)$.
\end{enumerate}
Moreover, there are equalities
\begin{equation}
  \label{eq:orderNR3}
|N(\Re)_3|= |\mathrm{C}(n,3, d)| = \binom{n+3d}{n}.
\end{equation}
\end{rmk}
We know that $\cA^{(d)}_3 = \cA_{3d}$, so $\dim\cA^{(d)}_3 = \dim \cA_{3d} =  \binom{n+3d}{n},$
which together with  (\ref{eq:orderNR3}) imply
\[|N(\Re)_3| = \dim \cA_{3d}.\]
It follows from
Lemma \ref{lem:diamondquadratic} that the set $\Re$ is a Gr\"{o}bner basis of the ideal $K$.
The set of leading monomials $\LM(\Re)$ is an antichain of monomials, hence $\Re$ is a minimal Gr\"{o}bner basis.
For $j > i$, every $F_{ji} \in \Re$ defined in \eqref{eq:Fji} is in normal form modulo $\Re \setminus \{F_{ji}\}$.
Similarly, for $(i,j) \in
 \Mv(n,d)$,  every $F_{ij} \in \Re$  defined in \eqref{eq:Fij} is in normal form
modulo $\Re \setminus \{F_{ij}\}$. We have proven that $\Re$ is a reduced Gr\"{o}bner basis of the ideal $K$.

 It follows from Remark \ref{rmk:LMremark1} that $\Re^{\prime}$ is a minimal Gr\"{o}bner basis of $K$.
 \end{proof}

\subsection{The Veronese map $v_{n,d}$ and the reduced Gr\"{o}bner basis of its kernel}
\label{subsec:grobner_veronese_map}
\begin{thm}
    \label{thm:Veronese_ker}
Let $n, d\in\N$ and $N=\binom{n+d}{d}-1$. Let
$\cA^n_\q$ be a quantum space defined by an $(n +1)\times (n +1)$ deformation matrix $\q$ and let $\cA^N_\g$ be the quantum space
 whose \mas $(N+1)\times (N+1)$ matrix $\g$ is determined by Lemma \ref{lem:Ver_well-defined1}.
 Let
\[v_{n,d} : \cA^N_{\g} \rightarrow  \cA^n_{\q}\]
be the Veronese map extending the assignment \[y_0 \mapsto w_0, \ y_1 \mapsto w_1, \ \dotsc ,\  y_N  \mapsto w_N.\]
\begin{enumerate}
\item
The image of $v_{n,d}$ is the d-Veronese subalgebra $(\cA^n_\q)^{(d)}$ of  $\cA^n_\q$.
\item
The kernel  $\mathfrak{K} := \ker (v_{n,d})
$  of the Veronese map has a reduced Gr\"{o}bner basis
consisting of exactly $\binom{N+2}{2}- \binom{n+2d}{n}$ binomials:
\begin{equation}
\label{eq:ker_Ver}
\cR_\q^{\rV}:=\{y_i y_j - \varphi_{ij} y_{i^{\prime}}y_{j^{\prime}} \mid (i,j) \in   \Mv (n,d),   (i^{\prime}, j^{\prime})\in \mathrm{C}(n,2,d), \varphi_{ij}
\in \C^{\times}\},
\end{equation}
where $\Nor(v_{n,d}(y_i y_j)) = \varphi_{ij} v_{n,d}(y_{i^{\prime}}y_{j^{\prime}})$,
 $y_iy_j \succ
y_{i^{\prime}}y_{j^{\prime}}$,
    and  $\varphi_{ij} \in \C^{\times}$ are invariants of
        $(\cA^n_\q)^{(d)}$
            given in  Notation \ref{notaz: data}.
\end{enumerate}
\end{thm}

\begin{proof}
Part (1) follows from Lemma \ref{lem:Ver_well-defined1}. For part (2), we first prove that the set $\cR_\q^{\rV}$ generates $\mathfrak{K}$. The proof is similar to the argument describing the kernel $K = \ker V$ in Theorem
\ref{thm:preVeronese_ker}.

Note that $\cR_\q^{\rV}\subset \mathfrak{K}$. Indeed,
by direct computation, one shows that  $v_{n,d}(\cR_\q^{\rV}) = \cR_2,$ the set of relations of the d-Veronese $(\cA^n_\q)^{(d)}$ given in (\ref{eq:fij}),
so $\cR_\q^{\rV}\subset \mathfrak{K}$.
Moreover, it follows from (\ref{eq:ker_Ver}) that for each pair $(i,j) \in   \Mv (n,d)$ the set $\cR_\q^{\rV}$ contains exactly one element, namely
$y_i y_j - \varphi_{ij} y_{i^{\prime}}y_{j^{\prime}}$, where $\Nor  (v_{n,d}(y_i y_j)) = \varphi_{ij} v_{n,d}(y_{i^{\prime}}y_{j^{\prime}})$. Here we consider the
normal form
$\Nor (v_{n,d}(y_i y_j)) = \Nor  (w_iw_j) = \varphi_{ij} w_{i^{\prime}}w_{j^{\prime}}$, see Theorem \ref{thm:d-Veronese_relations}(\ref{thm:d-Veronese_relations2}).
Hence 
\begin{equation}
\label{eq: card_cR}
|\cR_\q^{\rV}|= |\Mv (n,d)|= \binom{N+2}{2} -\binom{n+2d}{n},
\end{equation}
where the last equality follows from Lemma \ref{lem:orders}. By Convention \ref{rmk:conventionpreliminary}, we identify $\cA^N_\g \cong (\C \cT(Y_N), \bullet)$. Our goal is to show that the two set of normal words $N(\mathfrak{K})$ and  $N(\cR_\q^{\rV})$ coincide, where
\[N (\mathfrak{K})= N_{\prec_0}(\mathfrak{K})\subset \C \cT(Y_N), \quad \text{and} \quad N (\cR_\q^{\rV})= N_{\prec_0} (\cR_\q^{\rV})\subset \C \cT(Y_N),\]
as in Definition \ref{dfn:normal_terms}.
There are obvious isomorphisms of vector spaces
\[\cA^N_\g = \C\cT(Y_N) = \mathfrak{K} \oplus \C N(\mathfrak{K}).\]
For simplicity of notation, we set $B = \cA^N_\g / \mathfrak{K}$ and consider the canonical grading of $B$ induced by the grading of $\cA^N_\g$.
Then
\[B = 
\cA^N_\g/\ker (v_{n,d}) \cong \im (v_{n,d}) = (\cA^n_\q)^{(d)},\]
so there are equalities
\begin{equation}
\label{eq:direct_sum_graded}
(\cA^N_\g)_m = (\C\cT(Y_N))_m = (\mathfrak{K})_m \oplus (\C N(\mathfrak{K}))_m\mbox{ and } B_m \cong (\cA^n_\q)^{(d)}_m =(\cA^n_\q)_{md},
\end{equation}
for every $m\ge 2$. In particular, for $m =2$ one has $B_2 \cong (\cA^n_\q)^{(d)}_2 =(\cA^n_\q)_{2d}$ and
\[\dim (\cA^N_\g)_2 = \dim (\mathfrak{K})_2 + \dim (\cA^n_\q)_{2d}\mbox{, hence } \binom{N+2}{2}= \dim (\mathfrak{K})_2 + \binom{n+2d}{2},\]
 which implies
\[\dim (\mathfrak{K})_2 = \binom{N+2}{2} - \binom{n+2d}{2} = |\cR_\q^{\rV}|.  \]
It is clear that the set $\cR_\q^{\rV}$ is linearly independent, so it is a basis of the graded component $\mathfrak{K}_2$,
and $(\mathfrak{K})_2 = \C \cR_\q^{\rV}$. But the ideal $\mathfrak{K}$ is generated by homogeneous polynomials of degree $2$,
therefore
\begin{equation}
\label{eq:K-generators}
\mathfrak{K} = \mathfrak{K}_2= (\cR^{\rV}_{\q}),
\end{equation}
so $\cR_\q^{\rV}$ generates the kernel $\mathfrak{K}$.

We are now ready to prove that $\cR_\q^{\rV}$ is a Gr\"{o}bner basis of $\mathfrak{K}$. We shall provide two proofs.

\medskip \noindent
\textsl{First proof}.
Here we use an analogue of Remark \ref{rmk:NRe} in the settings of a quantum space.
\begin{rmk}
    \label{rmk:NRe1} The following are equivalent:
\begin{enumerate}
\item $y_iy_jy_k  \in N(\cR^{\rV}_{\q})_3$;
\item $y_iy_j  \in N(\cR^{\rV}_{\q})_2$ and $y_jy_k  \in N(\cR^{\rV}_{\q})_2$;
\item $(i,j,k) \in \mathrm{C}(n,3, d)$.
\end{enumerate}
Moreover there are equalities
\begin{equation}
  \label{eq:orderNR31}
|N(\cR^{\rV}_{\q})_3|= |\mathrm{C}(n,3, d)| = \binom{n+3d}{n}.
\end{equation}
\end{rmk}
By (\ref{eq:direct_sum_graded}),  
$\dim B_3 = \dim \cA_{3d} =  \binom{n+3d}{n},$
which together with  (\ref{eq:orderNR31}) implies
\[|N(\cR^{\rV}_{\q})_3| =\dim B_3.\]
Now Lemma \ref{lem:quant_diamondquadratic} implies that $\cR^{\rV}_{\q}$  is a Gr\"{o}bner basis of the ideal
$\mathfrak{K} = \ker (v_{n,d})$.
It is clear that $\cR^{\rV}_{\q}$ is the reduced Gr\"{o}bner basis of $\mathfrak{K}$.

\medskip \noindent
\textsl{Second proof}.
 We shall use Theorem \ref{thm:preVeronese_ker} and ideas from \cite{Mo94}.
By (\ref{eq:K-generators}), we know  that the set $\cR^{\rV}_{\q}$ generates $\mathfrak{K}$.
Consider now the ideal  $\Nor^{-1} (\mathfrak{K})$ in  $\C\asY$.
It is easy to see that
\[\Nor^{-1}(\mathfrak{K}) = \mathfrak{J}_\g + (\cR^{\rV}_{\q})= (\cR_\g)  + (\cR^{\rV}_{\q}) = K,\]
where $K= \ker V$  is the kernel of the  epimorphism $V: \C\asY \rightarrow \cA^{(d)}$ from Theorem
\ref{thm:preVeronese_ker}.
Indeed, the polynomials in $\cR_\g$ and $\cR^{\rV}_{\q}$, considered as elements of the free associative algebra $\C\asY$,
satisfy
\[\cR_\g =\Re_1^{\prime} \mbox{ and }\cR^{\rV}_{\q}= \Re_2 ,\]
where $\Re_1^{\prime}$ and $\Re_2$ are the relations given
in Theorem  \ref{thm:preVeronese_ker},
see \eqref{eq:gij} and \eqref{eq:Fij}.
Hence by the same theorem, the set $\Re^{\prime} = \cR_\g \cup\cR^{\rV}_{\q}$ is a minimal Gr\"{o}bner basis of the ideal $K$.
Theorem
\ref{thm:preVeronese_ker} also implies that the disjoint union of quadratic relations  $\Re = \Re_1 \cup \Re_2$ is the reduced Gr\"{o}bner basis of $K$ in $\C\asY$.
It follows from \cite[Proposition 9.3(3)]{Mo94} that the intersection
\[G = \Re\cap \C N(\mathfrak{J}_\g)= \Re\cap \C N(\cR_\g) \]
is the reduced Gr\"{o}bner basis of the ideal $\mathfrak{K} = \ker (v_{n,d})$.
Moreover, we have $N(\mathfrak{J}_\g)= \C\cT(Y_N)$.  Then the obvious equalities
\[G= \Re\cap \C N(\mathfrak{J}_\g)  =  (\Re_1 \cup \Re_2)\cap \C\cT(Y_N) = \Re_2 = \cR_\q^{\rV}  \]
imply that $\cR_\q^{\rV}$
is the reduced Gr\"{o}bner basis of $\mathfrak{K}$.
\end{proof}
We remark that \cite[Proposition 9.3(4)]{Mo94} implies that the
set $\cR_\g \cup G = \Re_1^{\prime}\cup \Re_2$ is the reduced
 Gr\"{o}bner basis of the
ideal $K$. This fact agrees with Part (3) of our
Theorem~\ref{thm:preVeronese_ker}, proven independently.

\begin{cor}
The set of leading monomials for the Gr\"obner basis
$\cR^{\rV}_{\q}$ does not depend on the
deformation matrix $\q$ and equals
\[
\LM(\cR^{\rV}_{\q}) = \lbrace y_iy_j \, \vert \, (i,j)\in \Mv(n,d) \rbrace.
\]
\end{cor}

\section{Segre products and Segre maps}
\label{Sec:Segremap}
In this section we introduce and investigate
non-commutative analogues of the Segre embedding
$S_{n,m}: \mathbb{P}^n \times \mathbb{P}^m \to \mathbb{P}^{(n+1)(m+1)-1}.$
The main result of the section is Theorem \ref{thm:Segre_ker}, which describes explicitly the reduced Gr\"obner basis for the kernel of the non-commutative
Segre map.
We first recall the notion of Segre product of graded algebras, following
\cite[Section 3.2]{PoPo}.
\begin{dfn}
    Let \[
    R=\bigoplus_{k\in\N_0}R_k\mbox{ and } S=\bigoplus_{k\in\N_0}S_k\]
    be graded algebras. The \emph{Segre product} of $R$ and $S$ is the graded algebra
    \[R\circ S:=\bigoplus_{k\in\N_0}R_k\otimes S_k.\]
\end{dfn}

Clearly, the Segre product $R\circ S$ is a subalgebra of the tensor product algebra $R\otimes S$. Note that the embedding is not a graded algebra
morphism, as it doubles grading. The Hilbert function of $R\circ S$
    satisfies
    \[h_{R\circ S}(t)=\dim(R\circ S)_t=\dim(R_t\otimes S_t)=\dim(R_t)\cdot\dim(S_t)=h_R(t)\cdot h_S(t).\]

Given $n, m \in \N$, let
\[N := (n+1)(m+1)-1.\]
Let $\q$ and $\q^{\prime}$ be two \mas matrices of sizes $(n+1)\times (n+1)$ and $(m+1)\times (m+1)$, respectively, and
let $\cA^n_\q$ and $\cA^m_{\q^{\prime}}$ be the corresponding quantum spaces.
We shall construct a quantum space
$\cA^N_\g$
defined via
$N+1$ (double indexed) generators
\[Z_{nm}= \{z_{i\alpha}\mid i\in\{0,\dots,n\}, \alpha\in\{0,\dots,m\}\}\]
and an $(N+1)\times (N+1)$ \mas matrix $\g$ uniquely determined  by
$\q$ and $\q^{\prime}$.

\begin{convention}
\label{convention:orderingZ}
 We order the set $Z_{nm}$ using the lexicographic ordering on
the pairs of indices $(i, \alpha), \; 0 \leq i \leq n, \; 0 \leq
\alpha \leq m$, that is, $z_{i\alpha} \prec z_{j\beta}$ if and
only if either (a) $i < j$, or (b) $i = j,$ and $\alpha < \beta$.
Thus
\begin{equation}
  \label{eq:orderingZ}
Z_{nm} = \{z_{00} \prec z_{01} \prec \cdots \prec z_{0m} \prec
z_{10}\prec \cdots \prec z_{nm-1 } \prec z_{nm} \}.
\end{equation}
When no confusion arises, we write $Z$ for $Z_{nm}$. As usual, we
consider the free associative algebra $\C\asZ $ and fix the
degree-lexicographic ordering $\prec$  induced by
(\ref{eq:orderingZ}) on the free monoid $\asZ$.

In this section, we shall work simultaneously
with three disjoint sets of variables, $X= X_n$, $Y =Y_m$, and $Z
= Z_{nm}$. We shall use notation $\cT (X)= \cT^n$, $\cT(Y) =\cT^m$
and $\cT(Z)$ for the corresponding sets of ordered terms in
variables $X$, respectively $Y$, respectively $Z$. In particular,
the set $\cT(Z)$ of ordered monomials in $Z$ with respect to the
ordering (\ref{eq:orderingZ}) is
\[
\cT(Z) = \{z_{00}^{k_{00}}z_{01}^{k_{01}}\dots z_{10}^{k_{10}}\dots   z_{nm}^{k_{nm}}\mid k_{i\alpha} \in \N_0,
\; i\in\{0,\dots,n\},\; \alpha\in\{0,\dots,m\}\}.
\]
\end{convention}
As in Convention \ref{rmk:conventionpreliminary}, we identify $\cA^n_{\q}$ with $(\C\cT(X),\bullet)$ and
$\cA^m_{\q^{\prime}}$ with $(\C\cT(Y),\bullet)$.

\begin{rmk}
\label{rmk:tensorproduct} Consider the free associative algebra
  $\C\asXY= \C\langle x_0, \dotsc, x_n, y_0,\dotsc y_m\rangle$, generated by the disjoint union $X_n \sqcup Y_m$,
and the free monoid $\asXY = \langle x_0, \dotsc, x_n, y_0,\dotsc y_m\rangle$ with the canonical degree-lexicographic ordering $\prec$ extending
$x_0\prec x_1 \prec\dots x_n\prec y_0 \prec y_1 \prec \dots y_m$. Let \[\cR_{0}=\cR(\cA^n_\q \otimes \cA^m_{\q^{\prime}}) = \cR_{\q}\cup
\cR_{\q^{\prime}}\cup
\{y_{\alpha}x_i - x_iy_{\alpha}\mid i\in\{0,\dots,n\},\ \alpha\in\{0,\dots,m\}\}.\]
Then $\cR_{0}$ is the reduced Gr\"{o}bner basis of the two-sided ideal $(\cR_{0})$ of $\C\asXY$ and there is an isomorphism of algebras
\[\C\asXY/(\cR_0) \cong \cA^n_\q \otimes \cA^m_{\q^{\prime}}.\]
\end{rmk}

\begin{pro}
\label{pro:Segre_well-defined}
In notation as above, let $\cA^n_\q$, and $\cA^m_{\q^{\prime}}$ be quantum spaces and let
$N := (n+1)(m+1)-1$.
Then there exists a unique $(N+1)\times (N+1)$ \mas matrix
 $\g =\|g_{i\alpha, j\beta}\|$
such that
the assignment \[z_{i\alpha} \mapsto x_i \otimes y_{\alpha}, \quad  \mbox{for every}\  i\in\{0,\dots,n\}\mbox{ and every }\alpha\in\{0,\dots,m\},\] extends to a well-defined $\C$-algebra homomorphism
 \begin{equation}
\label{eq:map_Segre}
  s_{n,m} : \cA^N_\g \to  \cA^n_\q \otimes \cA^m_{\q^{\prime}}.
 \end{equation}
Moreover, the following conditions hold
 \begin{enumerate}
\item The quantum space $\cA^N_\g $ is presented as
\[
\cA^N_\g = \C\asZ / (\cR_{\g}),
\]
where
\begin{equation}
\label{eq:GB_Segre1}
 \cR_{\g} := \{z_{j\beta}z_{i\alpha}  - (g_{j\beta, i\alpha})z_{i\alpha}z_{j\beta}\mid z_{j\beta} \succ z_{i\alpha}, z_{j\beta}, z_{i\alpha}\in Z\}
 \end{equation}
is a reduced Gr\"{o}bner basis for the two-sided ideal $(\cR_{\g})$ in $\C\asZ$.
\item
There is an isomorphism of algebras $\cA^N_\g\cong (\C \cT(Z),  \bullet )$, where the multiplication $\bullet$ is defined as   $u \bullet v := \Nor_{\cR_{\g}}(uv)$.
\item
The image $s_{n,m} (\cA^N_{\g})$ is the Segre subalgebra $\cA^n_{\q} \circ \cA^m_{\q^{\prime}}$ of $\cA^n_{\q} \otimes \cA^m_{\q^{\prime}}$.
 \end{enumerate}
\end{pro}

We call the homomorphism $s_{n,m} $ the \emph{$(n,m)$-Segre map}.
\begin{proof}
Assume that there exists an $(N+1)\times (N+1)$ \mas matrix $\g$ such that
$s_{n,m}$ is a homomorphism of $\C$-algebras.
 Let $Z=Z_{nm}$ as above be the set of generators of $\cA^N_{\g}$. We compute $s_{n,m}(z_{i\alpha} z_{j\beta})$ in two different ways:
\begin{align*}
 s_{n,m}(z_{i\alpha} z_{j\beta})& = s_{n,m}(z_{i\alpha})s_{n,m}(z_{j\beta}) \\
                                 &=(x_i\otimes y_{\alpha}) (x_j \otimes y_{\beta})= (x_ix_j) \otimes (y_{\alpha}y_{\beta}) \\
s_{n,m}(z_{i\alpha} z_{j\beta}) & = s_{n,m} (g_{i\alpha, j\beta} (z_{j\beta} z_{i\alpha})) = g_{i\alpha, j\beta} s_{n,m} ( z_{j\beta} z_{i\alpha})\\
                                &=  g_{i\alpha, j\beta} s_{n,m} ( z_{j\beta}) s_{n,m} (z_{i\alpha})  =g_{i\alpha, j\beta} (x_jx_i\otimes y_{\beta}y_{\alpha})\\
                                                                & = g_{i\alpha, j\beta} q_{ji} q^{\prime}_{\beta \alpha} (x_ix_j) \otimes (y_{\alpha}y_{\beta}).
 \end{align*}
Therefore,
\[ (x_ix_j) \otimes (y_{\alpha}y_{\beta}) =(g_{i\alpha, j\beta} q_{ji} q^{\prime}_{\beta \alpha}) (x_ix_j) \otimes (y_{\alpha}y_{\beta})\]
for every $i,j\in\{0,\dots,n\}$ and every $\alpha,\beta\in\{0,\dots,m\}$.
 It follows that $\g = \left\| g_{i\alpha, j \beta}\right\|$ is a \mas matrix uniquely determined by the equalities
\begin{equation}
\label{eq:matrix_Segre2}
 g_{i\alpha, j \beta} = (q_{ji} q^{\prime}_{\beta \alpha})^{-1} = q_{ij} q^{\prime}_{\alpha \beta},
 \end{equation}
We remark that the matrix $\g$ is equal to the the Kronecker product $\q \otimes \q^{\prime}$ of the matrices $\q$ and $\q^{\prime}$.

Conversely, if $\g$ is the \mas matrix defined via (\ref{eq:matrix_Segre2}), then the Segre map (\ref{eq:map_Segre}) is a well-defined
algebra
homomorphism.
Conditions (1) and (2) follow straightforwardly from the discussion in Section 3, see Remarks \ref{rmk:Grbasis} and Convention \ref{rmk:conventionpreliminary}.
The Segre subalgebra $\cA^n_{\q} \circ\cA^m_{\q^{\prime}}$ is generated by the elements $x_i \otimes y_{\alpha}$ for $i\in\{0,\dots,n\}$ and $\alpha\in\{0,\dots,m\}$. By construction
$s_{n,m}(z_{i\alpha}) = x_i \otimes y_{\alpha}$, hence the image $s_{n,m} (\cA^N_{\g})$ is the Segre subalgebra $\cA^n_{\q} \circ \cA^m_{\q^{\prime}}$, which proves (3).
\end{proof}

As usual, we identify the quantum space $\cA^N_\g$ with $(\C \cT(Z),  \bullet )$, see Convention \ref{rmk:conventionpreliminary}.
\begin{rmk}
\label{rmk:segre_Product}
Being a Segre product, the algebra $\cA^n_{\q} \circ \cA^m_{\q^{\prime}} =s_{n,m}(\cA^N_{\g})$
 inherits various properties from the two algebras $\cA^n_{\q}$ and $\cA^m_{\q^{\prime}}$. In particular, since the latter are one-generated, quadratic, and Koszul, it follows
 from \cite[Proposition 3.2.1]{PoPo} that the algebra
$\cA^n_{\q} \circ \cA^m_{\q^{\prime}}$ is also one-generated, quadratic, and Koszul.
Clearly, the set
$\{x_i\otimes y_{\alpha}\mid i\in\{0,\dots,n\}, \alpha\in\{0,\dots,m\}\}$ of cardinality $N+1 = (n+1)(m+1)$ is a generating set of $\cA^n_{\q} \circ \cA^m_{\q^{\prime}}$.
\end{rmk}

\begin{lem}
\label{lem:product_segre}
The following equalities hold in the Segre product
$\cA^n_{\q} \circ \cA^m_{\q^{\prime}}$, for all $i,j,\alpha,\beta$, such that $0 \leq i <j \leq n$ and $0 \leq \alpha < \beta \leq m$:
\begin{equation}
  \label{eq:trivial_segre}
  (x_i\circ y_{\alpha})(x_j\circ y_{\beta})  =(x_ix_j) \circ (y_{\alpha}y_{\beta}).
    \end{equation}
    \begin{equation}
  \label{eq:rel_segre1}
    \begin{array}{ll}
    (x_j\circ y_{\beta}) (x_i\circ y_{\alpha}) &=q_{ji}q^{\prime}_{\beta\alpha}(x_ix_j) \circ (y_{\alpha}y_{\beta})
                                                =q_{ji}q^{\prime}_{\beta\alpha} (x_i\circ y_{\alpha})(x_j\circ y_{\beta}). \\
    (x_j\circ y_{\alpha})(x_i\circ y_{\beta}) &=q_{ji}q^{\prime}_{\alpha\beta}(x_ix_j)\circ (y_{\beta}y_{\alpha})=q_{ji}q^{\prime}_{\alpha\beta} (x_i\circ y_{\beta}) (x_j\circ y_{\alpha}). \end{array}
            \end{equation}
            \begin{equation}
\label{eq:kernel_segre}
    (x_i\circ y_{\beta})(x_j\circ y_{\alpha}) =x_ix_j\circ y_{\beta}y_{\alpha}
                                              =q^{\prime}_{\beta\alpha}  (x_ix_j) \circ (y_{\alpha}y_{\beta}) =q^{\prime}_{\beta\alpha}(x_i\circ y_{\alpha})(x_j\circ y_{\beta})
            \end{equation}
 \begin{equation}
  \label{eq:rel_segre3}
  \begin{array}{ll}
    (x_j\circ y_{\alpha})(x_i\circ y_{\alpha}) &=q_{ji}  (x_ix_j) \circ (y_{\alpha}y_{\alpha})
                                               =q_{ji}  (x_i\circ y_{\alpha})(x_j\circ y_{\alpha})\\
    (x_i\circ y_{\beta}) (x_i\circ y_{\alpha}) &=q^{\prime}_{\beta\alpha}(x_ix_i) \circ (y_{\alpha}y_{\beta})
                                               =q^{\prime}_{\beta\alpha} (x_i\circ y_{\alpha})(x_i\circ y_{\beta})\\
       \end{array}
\end{equation}

                                                                                        \end{lem}
\begin{rmk}
\label{rmk:rel_segre}
\begin{enumerate}
\item
The equalities given in Lemma \ref{lem:product_segre} imply the following explicit list
 of defining relations $ \cR_{\g} $ for the quantum space $\cA^N_{\g}$:
\begin{equation}
  \label{eq:rel_segre4}
    \begin{array}{lll}
    z_{j\beta}z_{i\alpha}  - q_{ji}q^{\prime}_{\beta\alpha} z_{i\alpha}z_{j\beta}&\in \cR_{\g} & \text{by (\ref{eq:rel_segre1})}\\
    z_{j\alpha}z_{i\beta}  - q_{ji}q^{\prime}_{\alpha\beta} z_{i\beta}z_{j\alpha}&\in \cR_{\g} & \text{by (\ref{eq:rel_segre1})}\\
    z_{j\alpha}z_{i\alpha} - q_{ji}z_{i\alpha}z_{j\alpha}                        &\in \cR_{\g} & \text{by (\ref{eq:rel_segre3})} \\
    z_{i\beta}z_{i\alpha}  -q^{\prime}_{\beta\alpha} z_{i\alpha}z_{i\beta}       &\in \cR_{\g} & \text{by (\ref{eq:rel_segre3})}
    \end{array},
    \end{equation}
    for every $0 \leq i <j \leq n$ and every $0 \leq \alpha <\beta \leq m$.
    \item
    The equalities (\ref{eq:kernel_segre}) imply that the following quadratic binomials in $\cA^N_{\g}$ are in the kernel  of the Segre map:
    \begin{equation}
  \label{eq:ker_segre}
    z_{i\beta}z_{j\alpha}  - q^{\prime}_{\beta\alpha} z_{i\alpha}z_{j\beta} \in \ker s(n,m),
     \end{equation}
                                                                                            for every $0 \leq i <j \leq n$ and every $0 \leq \alpha <\beta \leq m$.
                                                                                            \end{enumerate}
\end{rmk}
\begin{notaz}
We denote by $\Ms(n,m)$ the following collection of quadruples:
\begin{equation}
\label{eq:Ms(n,m)}
 \Ms(n,m)= \{ (i,j,\beta, \alpha) \mid 0 \leq i<j \leq n,  0 \leq \alpha < \beta \leq m\}.
 \end{equation}
\end{notaz}

\begin{lem}
\label{lem:cardinality_Ms}
The cardinality of $\Ms(n,m)$ is
\begin{equation}
\label{eq:cardinalityMs(n,m)}
\left|\Ms(n,m)\right| = \binom{n+1}{2}\binom{m+1}{2}.
\end{equation}
\end{lem}

\begin{proof}
Clearly, $|\{(i,j) \mid 0 \leq i<j \leq n\}| = \binom{n+1}{2}$. Moreover, for each fixed pair $(i,j)$,  $0 \leq i<j \leq n,$ the number of quadruples  $\{(i,j,\beta, \alpha)\mid 0 \leq \alpha < \beta \leq m \}$ is exactly $\binom{m+1}{2}$, which finishes the proof.
\end{proof}

 We keep the notation and conventions of this section, in particular we identify the quantum space $\cA^N_\g$ with $(\C \cT(Z),  \bullet )$.
 Recall that if $P \subset
\cA^N_\g$ is an arbitrary set, then $\LM(P)=\LM_{\prec_0}(P)$
denotes the set of leading monomials
\[\LM(P)=\{\LM_{\prec_0}(f) \mid f \in P\}.\]
 A monomial $T\in \cT (Z)$ \emph{is normal modulo} $P$ if it does not contain as a subword  any $u \in \LM(P)$.
 The set of all normal mod $P$ monomials in $\cT(Z)$ is denoted by $N_{\prec_0}(P)$, so
\[N_{\prec_0}(P)= \{T \in \cT(Z) \mid   T   \text{ is normal mod P} \}.\]
A criterion for a Gr\"{o}bner basis $F$ of an ideal  $\mathfrak{K}=(F)$ in $\cA^N_\g$ follows straightforwardly
as an analogue of Lemma \ref{lem:quadr_Grbasis}, in which we
 only replace $Y_N$ with the set of generators $Z$, and keep the remaining notation and assumptions.

\begin{thm}
\label{thm:Segre_ker}
The set
    \[\cR_{\q,\q^{\prime}}^{\rS}:=\{z_{i\beta}z_{j\alpha}-\q^{\prime}_{\beta \alpha} z_{i\alpha}z_{j\beta} \mid 0 \leq i < j \leq n,   0 < \alpha < \beta \leq m\}\subset \cA^N_{\g}\]
    consisting of $\binom{n+1}{2}\binom{m+1}{2}$ quadratic binomials
    is a reduced Gr\"{o}bner basis for the kernel of the Segre map
    \[s_{n,m} : \cA^N_{\g} \to  \cA^n_{\q} \otimes \cA^m_{\q^{\prime}}.\]
    \end{thm}
\begin{proof}
It is clear that $\left|\cR_{\q,\q^{\prime}}^{\rS}\right| = \left|\Ms(n,m)\right|=\binom{n+1}{2}\binom{m+1}{2}$.
We set
\[
\begin{array}{lc}
\mathfrak{K} = \ker s_{n,m},  &N(\mathfrak{K}) = N_{\prec_0}(\mathfrak{K}),\\
\cR =\cR_{\q,\q^{\prime}}^{\rS}, & N(\cR) = N_{\prec_0}(\cR).
\end{array}
\]
 By Remark \ref{rmk:rel_segre}(2),   $\cR  \subset  \mathfrak{K}.$
We claim that $\cR$ generates $\mathfrak{K}$ as a two-sided ideal of $\cA^N_\g$.

The image $s_{n,m}(\cA^N_{\g})$ is the Segre product $\cA^n_{\q} \circ \cA^m_{\q^{\prime}}$, which is a quadratic algebra, see
Remark
\ref{rmk:segre_Product}.
Therefore the kernel $\mathfrak{K}$ is generated by polynomials of degree two.
Moreover, there is an isomorphism of vector spaces \[\C N(\mathfrak{K}) \cong \cA^n_{\q} \circ \cA^m_{\q^{\prime}}.\]
In particular,
\begin{equation}
\label{eq:dim1}
\dim (\C N(\mathfrak{K}))_2 = \dim ((\cA^n_{\q})_2) \dim ((\cA^m_{\q^{\prime}})_2) = \binom{n+2}{2}\binom{m+2}{2}.
\end{equation}
 It is clear that
$(\cA^N_\g)_2 = (\C\cT(Z))_2 = (\mathfrak{K})_2 \oplus (\C N(\mathfrak{K}))_2$,
hence
\begin{align*}
\label{eq:dimensions}
\dim (\mathfrak{K})_2 &= \dim(\cA^N_\g)_2 - \dim (\C N(\mathfrak{K}))_2= \binom{N+2}{2} - \binom{n+2}{2}\binom{m+2}{2}\nonumber\\
&= \binom{(n+1)(m+1)+1}{2} - \binom{n+2}{2}\binom{m+2}{2}= \binom{n+1}{2}\binom{m+1}{2}=  \left| \cR \right|.
\end{align*}
Now the equality $|\cR| = \dim (\mathfrak{K})_2$, together with the obvious linear independence of the elements of $\cR$, imply that
$\cR$ is a $\C$-basis of $(\mathfrak{K})_2$, so it spans the space
$(\mathfrak{K})_2$.
But we know that the kernel $\mathfrak{K}$ is generated by polynomials of degree two, hence $\mathfrak{K}= (\cR)$.

Next we shall prove that $\cR$ is a Gr\"{o}bner basis of the ideal $\mathfrak{K}$.
Let  $B = \cA^N/\mathfrak{K}$. Then
\[B 
= \cA^N/\ker(s_{n,m}) \cong s_{n,m}(\cA^N) = \cA^n_{\q} \circ \cA^m_{\q^{\prime}}.\]
Hence
\begin{equation}
\label{eq:dimB3}
\dim B_3 = \dim (\cA^n_{\q} \circ \cA^m_{\q^{\prime}})_3 =  \dim (\cA^n_{\q})_3 {\cdot} \dim(\cA^m_{\q^{\prime}})_3 = \binom{n+3}{3}\binom{m+3}{3}.
\end{equation}

We claim that $\dim B_3 = |(N(\cR))_3|$.
Indeed, by the identification $\cA^N_\g \simeq (\C \cT(Z),  \bullet )$ we have
\[(\cA^N_\g)_3 = (\C\cT(Z))_3 =
\C\{z_{i\alpha}z_{j\beta}z_{k\gamma}\mid (i,\alpha) \leq (j,\beta)\leq (k,\gamma),   0\leq i,j,k \leq n, 0\leq \alpha,\beta,\gamma\leq m \}.\]
Clearly, a monomial $z_{i\alpha}z_{j\beta}z_{k\gamma}\in  (\cT(Z))_3$ is normal modulo $\cR$ if and only if each of its subwords of length 2,
$z_{i\alpha}z_{j\beta}$ and $z_{j\beta}z_{k\gamma}$, is normal modulo $\cR$. Moreover,
\[N(\cR)_2=
\{z_{i\alpha}z_{j\beta}\mid 0 \leq i \leq j\leq n, 0\leq \alpha\leq\beta\leq m \},\]
therefore
\begin{equation}
\label{eq:the_set_N3}
N(\cR)_3 = \{z_{i\alpha}z_{j\beta}z_{k\gamma}\mid 0 \leq i\leq j \leq k \leq n,  \quad 0 \leq \alpha \leq \beta\leq \gamma\leq m \}.
\end{equation}
It follows from (\ref{eq:the_set_N3})
that
\[
\left|N(\cR)_3\right| =  \binom{n+3}{3}\binom{m+3}{3},
\]
which together with (\ref{eq:dimB3}) give the desired equality $\dim B_3 = |(N(\cR))_3|$.
Now Lemma
\ref{lem:quant_diamondquadratic}  implies that $\cR$ is a Gr\"{o}bner basis of the ideal $\mathfrak{K}$.
It is obvious that $\cR$ is a reduced Gr\"{o}bner basis of $\mathfrak{K}$.
\end{proof}
\section{Examples}
\label{sec:examples}
We present here some example that illustrate the results of our paper.

\begin{ex}[The non-commutative twisted cubic curve]\label{ex:ver11mod}
Let $n=1$ and $d=3$. Then
\[
\begin{array}{l}
X= \{x_0, x_1\},\quad  \q= \begin{pmatrix}
1 & q^{-1} \\
q & 1
\end{pmatrix},\quad
\cA^1_\q = \C\langle x_0, x_1\rangle / (x_1x_0 -qx_0x_1).
 \end{array}
 \]
 In this case $N=\binom{1+3}{3}-1 =3$ and the corresponding quantum space $\cA^3_\g$ is defined by the following data
\[
\begin{array}{l}
 Y= \{y_0, y_1, y_2, y_3\}, \quad
 \q= \begin{pmatrix}
1 &q^{-3} &q^{-6}&q^{-9} \\
q^3 &1 &q^{-3}&q^{-6} \\
q^6 &q^3 &1&q^{-3} \\
q^9 &q^6 &q^3&1 \\
\end{pmatrix}.
 \end{array}
 \]
 The kernel $\ker(v_{1,3})$ of the  Veronese map $v_{1,3}: \cA^3_\g \rightarrow \cA^1_\q$
 has a reduced Gr\"{o}bner basis $G$ given below
     \[G= \{y_1^2-q^2y_0y_2, \, y_1y_2-qy_0y_3, \, y_2^2-q^2y_1y_3\}.\]
     We have used the fact that in this case $\Mv(1,3)=\{(1,1),(1,2),(2,2)\}$.

Setting $q=1$ we obtain that the defining ideal for the \emph{commutative} Veronese is generated by the three polynomials
     $\{y_1^2-y_0y_2,y_1y_2-y_0y_3, y_2^2-y_1y_3\}$.
This is exactly the set of generators described and discussed in \cite[pp. 23, 51]{Harris}.
\end{ex}

\begin{ex}[The non-commutative rational normal curves]\label{ex:ver1d}
Generalising the previous example, we consider $n=1$ and $d$ arbitrary. In notation as above, we write \[
\cA^1_\q = \C\langle x_0, x_1\rangle / (x_1x_0 -qx_0x_1).
 \]
In this case, $N= \binom{d+1}{d}-1=d$ and the corresponding quantum space $\mathcal{A}^d_{\g}$ is determined by the data
\begin{equation}
\label{eq:data_1d}
\begin{array}{l}
 Y= \{y_0, y_1, \dots, y_d\}, \quad
 \q= \begin{pmatrix}
1 &q^{-d} &q^{-2d}& \dots & \dots & \dots & q^{-d^2} \\
q^{d} &1 &q^{-d }& & &  & q^{-d(d-1)} \\
q^{2d} &q^{d} &1& \ddots && & q^{-d(d-2)} \\
\vdots &  & \ddots & \ddots & \ddots && \vdots  \\
\vdots & & & \ddots & \ddots & \ddots & \vdots &\\
q^{d(d-1)} & & & & \ddots & 1 & q^{-d} \\
q^{d^2} &q^{d(d-1)} &q^{d(d-2)} &\dots & \dots & q^d & 1 \\
\end{pmatrix}.
 \end{array}
\end{equation}
 Observe that whenever $q$ is a $d$-th root of unity, the derived $(1,d)$-quantum space is a commutative algebra.

 The kernel $\ker(v_{1,d})$ of the  Veronese map $v_{1,d}: \cA^d_\g \rightarrow \cA^1_\q$
 has a reduced Gr\"{o}bner basis $G$ given by $\binom{d}{2}$ quadratic relations:
     \[G= \lbrace y_i y_j-h_{ij} \ \vert \  1 \leq i \leq j \leq d-1 \rbrace, \quad
h_{ij} = \begin{cases} q^{i(d-j)} y_0 y_{i+1} & i+j \leq d \\
q^{i(d-j)}y_{i+j-d} y_d & i+j > d
\end{cases}.\]

Once again, for $q=1$ we obtain that a reduced Gr\"obner basis for the defining ideal of the \emph{commutative} rational normal curve (see \cite[Example 1.16]{Harris}).
\end{ex}

\begin{ex}[The non-commutative Veronese surface]\label{ex:ver22}
Let $n=d=2$, that is,
\begin{gather*}
X = \{x_0, x_1, x_2\},\quad  \q= \begin{pmatrix}
1 & q_{10}^{-1} & q_{20}^{-1} \\
q_{10} & 1 & q_{21}^{-1} \\
q_{20} & q_{21} & 1
\end{pmatrix},\\
\cA^2_\q = \C\langle x_0, x_1, x_2\rangle / (x_1x_0 -q_{10}x_0x_1, x_2x_0 -q_{20} x_0x_2, x_2x_1-q_{21} x_1x_2).
\end{gather*}
 In this case $N=5$ and the corresponding $(2,2)$-quantum space $\cA^5_\g$ is completely determined by the data
\[
\begin{array}{l}
 Y= \{y_0, y_1, y_2, y_3, y_4, y_5\}, \quad
 \g := \begin{pmatrix}
        1 & q_{10}^{-2} & q_{20}^{-2} & q_{10}^{-4} & q_{20}^{-2}q_{10}^{-2} & q_{20}^{-4}\\
        q_{10}^{2} & 1 & q_{20}^{-1}q_{21}^{-1}q_{10} & q_{10}^{-2} & (q_{10}q_{20}q_{21})^{-1} & q_{20}^{-2} q_{21}^{-2}  \\
        q_{20}^{2} & q_{20}q_{21}q_{10}^{-1} & 1 & q_{21}^2 q_{10}^{-2} & q_{21}q_{10}^{-1}q_{20}^{-1} & q_{20}^{-2}\\
        q_{10}^{4} & q_{10}^{2} & q_{21}^{-2}q_{10}^{2} & 1 & q_{21}^{-2} & q_{21}^{-4}\\
        q_{20}^{2}q_{10}^{2} & q_{10}q_{20}q_{21} & q_{21}^{-1}q_{10}q_{20} & q_{21}^{2} & 1 & q_{21}^{-2}\\
        q_{20}^{4} & q_{20}^{2}q_{21}^{2} & q_{20}^{2} & q_{21}^{4} &  q_{21}^{2} & 1
\end{pmatrix}
 \end{array}
 \]

Observe that inside the matrix $\g$ we find as submatrices three occurrences of the matrix in \eqref{eq:data_1d} for $d=2$ and $q$ equal to one of the three commutation parameters, namely
\[\begin{pmatrix}
        1 & q_{10}^{-2} & q_{10}^{-4} \\
        q_{10}^{2} & 1 & q_{10}^{-2} \\
        q_{10}^{4} & q_{10}^{2} &  1
\end{pmatrix}, \qquad  \begin{pmatrix}
        1 & q_{20}^{-2} & q_{20}^{-4} \\
        q_{20}^{2} & 1 & q_{20}^{-2} \\
        q_{20}^{4} & q_{20}^{2} &  1
\end{pmatrix}, \quad \mbox{and} \quad \begin{pmatrix}
        1 & q_{21}^{-2} & q_{21}^{-4} \\
        q_{21}^{2} & 1 & q_{21}^{-2} \\
        q_{21}^{4} & q_{21}^{2} &  1
\end{pmatrix}.
\]

The kernel of the Veronese map $v_{2,2} : \cA^5_\g \to  \cA^2_\q$ has a reduced Gr\"obner basis consisting of six quadratic polynomials
\begin{align*}
G = (
 & y_1^2 - q_{10} y_0y_3,  \, y_1y_2-q_{10}y_0y_4, \,  y_2^2 - q_{20}y_0y_5,\\
 & y_2y_3 - q_{21}^2 y_1y_4, \, y_2y_4 -q_{21}y_1y_5, \,  y_4^2 -q_{21}y_3y_5 ) .
\end{align*}
\end{ex}

\begin{ex}[The Segre quadric]
Let $n=m=1$. Following the above conventions, we write

$\cA^1_\q = \C\langle x_0, x_1\rangle / (x_1x_0 -qx_0x_1)$ and  $\cA^1_{\q^\prime} = \C\langle y_0, y_1\rangle / (y_1y_0 -q^\prime y_0y_1)$.

In this case, $N=3$ and the quantum space $\mathcal{A}^3_{\g}$ is determined by the data
\[\begin{array}{l}
 Z= \{z_{00}, z_{01}, z_{10}, z_{11} \}, \quad
 \g= \begin{pmatrix}
1 &{q^\prime}^{-1} & q^{-1} & (q^\prime q)^{-1}\\
q^{\prime} & 1 & q^{-1} q^{\prime} & q^{-1} \\
q & q(q^\prime)^{-1} & 1 & (q^\prime)^{-1} \\
qq^{\prime} & q & q^{\prime} & 1 \\
\end{pmatrix}.
 \end{array}
 \]
 The kernel $\ker (s_{1,1})$ of the Segre map $s_{1,1} : \cA^3_{\g} \to \cA^1_{\q} \otimes \cA^1_{\q^\prime}$ has a reduced Gr\"obner basis consisting of a single quadratic polynomial
\[ G= \lbrace z_{01}z_{10} - q^\prime z_{00}z_{11} \rbrace.\]
\end{ex}

\begin{ex}[The non-commutative Segre threefold]
Let $n=2$ and $m=1$. We consider \[
\cA^2_\q = \C\langle x_0, x_1, x_2\rangle / (x_1x_0 -q_{1,0}x_0x_1, x_2x_0 -q_{2,0} x_0x_2, x_2x_1-q_{2,1} x_1x_2)\]
and  \[ \cA^1_{\q^\prime} = \C\langle y_0 y_1\rangle / (y_1y_0 -q^\prime y_0y_1).\]
Then $N =5$ and the corresponding $(2,1)$-derived quantum space is determined by the following data:
 \[\begin{array}{l}
 Z= \{z_{00}, z_{01}, z_{10}, z_{11}, z_{20}, z_{21} \}, \quad
 \g= \begin{pmatrix}
1 & (q^\prime)^{-1} & q_{10}^{-1} & (q_{10}q^\prime)^{-1} & q_{20}^{-1} & (q_{20}q^\prime)^{-1}\\
q^{\prime} & 1 & q_{10}^{-1} q^{\prime} & q_{10}^{-1} & q_{20}^{-1} q^{\prime} & q_{20}^{-1}\\
q_{10} & q_{10}(q^\prime)^{-1} & 1 & (q^\prime)^{-1} & q_{21}^{-1} & (q_{21}q^\prime)^{-1} \\
q_{10}q^{\prime} & q_{10} & q^{\prime} & 1 & q_{21}^{-1}q^{\prime} & q_{21}^{-1}\\
q_{20} & q_{20}(q^\prime)^{-1} & q_{21} & q_{21}(q^\prime)^{-1} & 1 & (q^{\prime})^{-1}\\
q_{20}q^\prime & q_{20} & q_{21} q^\prime & q_{21} & q^{\prime} & 1 \\
\end{pmatrix}.
 \end{array}
 \]
 The kernel $\ker (s_{2,1})$ of the Segre map $s_{2,1} : \cA^5_{\g} \to \cA^2_{\q} \otimes \cA^1_{\q^\prime}$ has a reduced Gr\"obner basis consisting of three quadratic polynomials
\[ G= \lbrace z_{01}z_{10} - q^\prime z_{00}z_{11}, z_{01}z_{20} - q^\prime z_{00}z_{21}, z_{11}z_{20} - q^\prime z_{10}z_{21} \rbrace.\]
\end{ex}


\begin{thebibliography}{99}
\bibitem{AFS}
C. Am\'{e}ndola, J.-C. Faug\`ere\ and\ B. Sturmfels, Moment varieties of Gaussian mixtures, J. Algebr. Stat. {\bf 7} (2016), no.~1, 14--28. \MR{3529332}

\bibitem{AS}
M. Artin\ and\ W. F. Schelter, Graded algebras of global dimension $3$, Adv. in Math. {\bf 66} (1987), no.~2, 171--216. \MR{0917738}






 \bibitem{AuKaOr08}
 D. Auroux, L. Katzarkov\ and\ D. Orlov, { Mirror symmetry for weighted projective planes and their noncommutative deformations}, Ann. of Math. (2) {\bf 167} (2008), no.~3,
 867--943. \MR{2415388}


\bibitem{BeZy}
I. Bengtsson\ and\ K. \.{Z}yczkowski, {\it Geometry of quantum states}, Cambridge University Press, Cambridge, 2017. \MR{3752196}

\bibitem{Bergman}
G. M. Bergman, {The diamond lemma for ring theory}, Adv. in Math. {\bf 29} (1978), no.~2, 178--218. \MR{0506890}




\bibitem{CiLaSz11}
L. Cirio, G. Landi\ and\ R. J. Szabo, {Algebraic deformations of toric varieties II: noncommutative instantons}, Adv. Theor. Math. Phys. {\bf 15} (2011), no.~6, 1817--1907.
\MR{2989814}

 \bibitem{CiLaSz13}
L. S. Cirio, G. Landi\ and\ R. J. Szabo, {Algebraic deformations of toric varieties I. General constructions}, Adv. Math. {\bf 246} (2013), 33--88. \MR{3091799}





\bibitem{EPS}
D. Eisenbud, I. Peeva\ and\ B. Sturmfels, Non-commutative Gr\"{o}bner bases for commutative algebras, Proc. Amer. Math. Soc. {\bf 126} (1998), no.~3, 687--691. \MR{1443825}

\bibitem{GaStiSturm05}
L. D. Garcia, M. Stillman\ and\ B. Sturmfels, Algebraic geometry of Bayesian networks, J. Symbolic Comput. {\bf 39} (2005), no.~3-4, 331--355. \MR{2168286}


\bibitem{GI91}  T. Gateva-Ivanova, On the Noetherianity of some associative finitely presented algebras, J. Algebra {\bf 138} (1991), no.~1, 13--35. \MR{1102566}

\bibitem{GI91A}
T. Gateva-Ivanova, Noetherian properties and growth of some associative algebras, in {\it Effective methods in algebraic geometry (Castiglioncello, 1990)}, 143--158, Progr.
Math., 94, Birkh\"{a}user Boston, Boston, MA. \MR{1106419}

 \bibitem{Ga94}  T. Gateva-Ivanova, Noetherian properties of skew polynomial rings with binomial relations, Trans. Amer. Math. Soc. {\bf 343} (1994), no.~1, 203--219.
     \MR{1173854}

\bibitem{Ga96}  T. Gateva-Ivanova, Skew polynomial rings with binomial relations, J. Algebra {\bf 185} (1996), no.~3, 710--753. \MR{1419721}

\bibitem{GIVB98}T. Gateva-Ivanova\ and\ M. Van den Bergh, Semigroups of $I$-type, J. Algebra {\bf 206} (1998), no.~1, 97--112. \MR{1637256}

\bibitem{GI04} T. Gateva-Ivanova, Binomial skew-polynomial rings, Artin--Schelter regular rings, and
   binomial solutions of the Yang--Baxter equation, Serdica Mathematical Journal, \textbf{30 }(2004), 431--470.


 \bibitem{Ga12}
 T. Gateva-Ivanova, {Quadratic algebras, Yang-Baxter equation, and Artin-Schelter regularity}, Adv. Math. {\bf 230} (2012), no.~4-6, 2152--2175. \MR{2927367}




 \bibitem{Harris} J. Harris, {\it Algebraic geometry}, Graduate Texts in Mathematics, 133, Springer-Verlag, New York, 1992. \MR{1182558}


 \bibitem{KRW90}
 A. Kandri-Rody\ and\ V. Weispfenning, {Noncommutative Gr\"{o}bner bases in algebras of solvable type}, J. Symbolic Comput. {\bf 9} (1990), no.~1, 1--26. \MR{1044911}

\bibitem{KaKuOr01}
A. Kapustin, A. Kuznetsov\ and\ D. Orlov, Noncommutative instantons and twistor transform, Comm. Math. Phys. {\bf 221} (2001), no.~2, 385--432. \MR{1845330}

\bibitem{Lan} J. M.
J. M. Landsberg, {\it Tensors: geometry and applications}, Graduate Studies in Mathematics, 128, American Mathematical Society, Providence, RI, 2012. \MR{2865915}

\bibitem{Latyshev}
V.N. Latyshev, Combinatorial ring theory. Standard bases,  Izd. Mosk. Univ., Moscow (1988).

\bibitem{Ma87}
Yu. I. Manin, Some remarks on Koszul algebras and quantum groups, Ann. Inst. Fourier (Grenoble) {\bf 37} (1987), no.~4, 191--205. \MR{0927397}

\bibitem{Ma88}
Yu. I. Manin, {\it Quantum groups and noncommutative geometry}, Universit\'{e} de Montr\'{e}al, Centre de Recherches Math\'{e}matiques, Montreal, QC, 1988. \MR{1016381}


 \bibitem{Mo94}
T. Mora, An introduction to commutative and noncommutative Gr\"{o}bner bases, Theoret. Comput. Sci. {\bf 134} (1994), no.~1, 131--173. \MR{1299371}



 \bibitem{PoPo}
 A. Polishchuk\ and\ L. Positselski, {\it Quadratic algebras}, University Lecture Series, 37, American Mathematical Society, Providence, RI, 2005. \MR{2177131}

\bibitem{priddy}
S. B. Priddy, {Koszul resolutions}, Trans. Amer. Math. Soc. {\bf 152} (1970), 39--60. \MR{0265437}

\bibitem{ShTi01}
B. Shelton\ and\ C. Tingey, {On Koszul algebras and a new construction of Artin-Schelter regular algebras}, J. Algebra {\bf 241} (2001), no.~2, 789--798. \MR{1843325}

 \bibitem{Sm03}
S. P. Smith, Maps between non-commutative spaces, Trans. Amer. Math. Soc. {\bf 356} (2004), no.~7, 2927--2944. \MR{2052602}


%
 \bibitem{V92}
 A. B. Ver\"{e}vkin, {\it On a non-commutative analogue of the category of coherent sheaves on a projective scheme}, in {\it Algebra and analysis (Tomsk, 1989)}, 41--53, Amer.Math. Soc. Transl. Ser. 2, 151, Amer. Math. Soc., Providence, RI. \MR{1191171}

\end{thebibliography}
\end{document}